\documentclass[a4paper, 11pt]{amsart}
\usepackage[scale=0.72,centering]{geometry}
\usepackage{amsmath,amssymb,amsthm,graphicx,bbm,pdfpages}
\usepackage{color,CJK,mathrsfs}
\usepackage[font={footnotesize}]{caption}
\numberwithin{equation}{section}
\newtheorem{theorem}{\textbf{Theorem}}[section]

\newtheorem{lemma}[theorem]{\textbf{Lemma}}

\newtheorem{proposition}[theorem]{\textbf{Proposition}}
\newtheorem{remark}[theorem]{\textbf{Remark}}
\begin{document}
\author{RAN WEI}
\thanks{Department of Mathematics, National University of Singapore, 10 Lower Kent Ridge Road, 119076 Singapore. Email address: weiran@u.nus.edu}
\title{On the Long-range Directed Polymer Model}
\date{24 August, 2016}
\begin{abstract}
We study the long-range directed polymer model on $\mathbbm{Z}$ in a random environment, where the underlying random walk lies in the domain of attraction of an $\alpha$-stable process for some $\alpha\in(0,2]$. Similar to the more classic nearest-neighbor directed polymer model, as the inverse temperature $\beta$ increases, the model undergoes a transition from a weak disorder regime to a strong disorder regime. We extend most of the important results known for the nearest-neighbor directed polymer model on $\mathbbm{Z}^d$ to the long-range model on $\mathbbm{Z}$. More precisely, we show that in the entire weak disorder regime, the polymer satisfies an analogue of invariance principle, while in the so-called very strong disorder regime, the polymer end point distribution contains macroscopic atoms and under some mild conditions, the polymer has a super-$\alpha$-stable motion. Furthermore, for $\alpha \in (1,2]$, we show that the model is in the very strong disorder regime whenever $\beta>0$, and we give explicit bounds on the free energy.
\end{abstract}
\maketitle
\section{Introduction}
The directed polymer model was first introduced by Huse and Henley \cite{huse1985pinning} in the study of the Ising model. A
later motivation for studying the model and generalize it in arbitrary dimension was the observation that when a polymer chain stretches in some media with impurity or charges, the behavior of the polymer chain will be influenced by the interaction between the polymer chain and the environment. The polymer chain is modelled by a directed random walk, and a family of random variables in space represents the random environment. The first mathematical study of the directed polymer model was due to Imbrie and Spencer \cite{imbrie1988diffusion}, which was then followed by many other authors e.g.\cite{berger2015high, bolthausen1989note, carmona2002partition, comets2003directed, comets2006directed, comets2006majorizing, lacoin2010new, vargas2007strong}. For an early review, see \cite{comets2004probabilistic}, and for a comprehensive introduction to directed polymer in random environment and other related polymer models, see \cite{den2009random}.

So far, most of the results achieved in the study of directed polymer are based on the assumption that the polymer chain performs a simple symmetric random walk. In that case, we call the model the nearest-neighbor directed polymer. It is natural to consider replacing the simple random walk by more complicated random walks to reflect a variety of physical phenomena. In \cite{MR2381596}, Comets considered long-range random walks, whose increment distribution is in the domain of attraction of some $\alpha$-stable law. The reason that we consider the long-range model is that it models superdiffusive motions, unlike the nearest-neighbor model, which only models the diffusive motion. Another reason that motivates the study of long-range directed polymer is that in recent years, long-range random walks have played an increasingly important role in related fields, such as mathematical finance and statistics. It is likely that the directed polymer model may be applied to the study of other subjects.

In \cite{MR2381596}, the author extended some early results for the nearest-neighbor directed polymer to the long-range case. Since then, much progress has been made for the nearest-neighbor model. The goal of this paper is to investigate whether these newer results can also be extended to the long-range model and what are the important differences between the two cases. We will see later that there are indeed some differences between the long-range model and the nearest-neighbor model due to the heavy-tailed increments, which will result in some technical difficulties.
\begin{remark}\label{T101}
In \cite{miura2008strong}, Miura, Tawara, and Tsuchida also studied a long-range model. They considered a continuous case in which the polymer chains are modelled by symmetric L\'evy processes and the random environment is given by a time-space Poisson point process. The continuous model is worth investigating so we mention this reference here for literature completeness but we will focus on discrete model in this paper.
\end{remark}
\subsection{Long-range directed polymer model}
Let $S=(S_{n})_{n\geq0}$ be a heavy-tailed random walk on $\mathbbm{Z}$ with i.i.d. increments, starting at $0$. The law of $S$ is denoted by $\mathbf{P}$ and the corresponding expectation is denoted by $\mathbf{E}$. We assume that the increment distribution of $S$ is in the domain of attraction of some stable law, which is equivalent to
\begin{equation}\label{101}
\mathbf{P}(|S_{1}|\geq n)=n^{-\alpha}L(n),~\forall n\geq1,~\mbox{for some}~\alpha\in(0,2),
\end{equation}
or
\begin{equation}\label{102}
\mathbf{E}\left[(S_{1})^{2}\mathbbm{1}_{\{|S_{1}|\leq n\}}\right]=L(n),~\forall n\geq1,~\mbox{for}~\alpha=2,
\end{equation}
where $L(\cdot)$ is some positive function slowly varying at infinity (see \cite[Theorem 3.2]{gut2012probability} and \cite[Chapter 1]{bingham1989regular}). Under the condition \eqref{101} or \eqref{102}, the random walk $S$ converges to some $\alpha$-stable law after centering and scaling, that is, we can find a sequence of centering factors $\{b_{n}\}_{n\in\mathbbm{N}}$ and a sequence of scaling factors $\{a_{n}\}_{n\in\mathbbm{N}}$, such that
\begin{equation}\label{103}
\frac{S_{n}-b_{n}}{a_{n}}\Rightarrow X_{\alpha}~\mbox{weakly as}~n\to\infty,
\end{equation}
where $X_{\alpha}$ is some stable law with stable exponent $\alpha\in(0,2]$. The scaling factor $a_{n}$ is determined by the stable exponent $\alpha$ and the slowly varying function $L(x)$, and can be expressed as $n^{\frac{1}{\alpha}}l(n)$ for some slowly varying function $l(n)$. When $\alpha\in(0,1)$, the centering factors $b_{n}$ can be chosen as $0$. When $\alpha\in(1,2]$, $\mathbf{E}[S_{1}]$ exists and $b_{n}$ can be chosen as $n\mathbf{E}[S_{1}]$. For simplicity, we just assume $\mathbf{E}[S_{1}]=0$ and we will see that all proof can be adapted
in a straight-forward manner for the non-zero mean case. When $\alpha=1$, $b_{n}$ can be computed, but we set $b_{n}=0$ for technical reason. For details, see \cite[Chapter 7]{newell1955limit}. Therefore, throughout this paper, we assume $b_{n}=0$.

We assume from now on that the random environment is described by a family of i.i.d. random variables $\omega=(\omega_{i,x})_{(i,x)\in\mathbbm{N}\times\mathbbm{Z}}$, which is independent of the random walk $S$. The law of $\omega$ is denoted by $\mathbbm{P}$ and the corresponding expectation is denoted by $\mathbbm{E}$. We also assume that the random environment has a finite logarithmic moment generating function, at least for small enough $|\beta|$,
\begin{equation}\label{104}
\lambda(\beta):=\log\mathbbm{E}[\exp(\beta\omega_{i,x})]<\infty,~~\forall\beta\in[-c,c],~\mbox{for some}~c>0.
\end{equation}
Without loss of generality, we can further assume that $\mathbbm{E}[\omega_{i,x}]=0$ and $\mathbbm{E}[(\omega_{i,x})^{2}]=1$.

Given the random environment $\omega$, for any $N\geq0$, and $\beta>0$, we can define the polymer measure through Gibbs transformation of the law $\mathbf{P}$ of the random walk up to time $N$ by
\begin{equation}\label{105}
\frac{\mbox{d}\mathbf{P}_{N,\beta}^{\omega}}{\mbox{d}\mathbf{P}}(S):=\frac{1}{Z_{N,\beta}^{\omega}}\exp\left(\sum\limits_{n=1}^{N}\beta\omega_{n,S_{n}}\right),
\end{equation}
where
\begin{equation}\label{106}
Z_{N,\beta}^{\omega}=\mathbf{E}\left[\exp\left(\sum\limits_{n=1}^{N}\beta\omega_{n,S_{n}}\right)\right]
\end{equation}
is the partition function which makes $\mathbf{P}_{N,\beta}^{\omega}$ a probability measure and $\beta$ is the inverse temperature. We also denote the Hamiltonian of the system by
\begin{equation}\label{107}
H_{N}^{\omega}(S):=-\sum\limits_{n=1}^{N}\omega_{n,S_{n}},
\end{equation}
which represents the energy of the path of the random walk. It can be seen from \eqref{105} that under the polymer measure $\mathbf{P}_{N,\beta}^{\omega}$, the random walk paths with low energy carry more weights.
\begin{remark}\label{T102}
Unlike many other papers concerning the nearest-neighbor model, in this paper, we only consider the model on $\mathbbm{Z}^{1+1}$ instead of $\mathbbm{Z}^{d+1}$. The reason is that when we later consider the significant classification of the strong disorder regime and the weak disorder regime, whether the random walk is recurrent or transient plays a key role, see \cite{MR2381596, comets2003directed}. It is known that for heavy-tailed random walks on $\mathbbm{Z}$ satisfying \eqref{101} or \eqref{102}, the random walk is recurrent for $\alpha\in(1,2]$ and transient for $\alpha\in(0,1)$, and for the critical case $\alpha=1$, whether the random walk is recurrent or transient depends on the slowly varying function $L(x)$. In dimension 2, the random walk is transient for $\alpha\in(0,2)$. And in higher dimension, the random walk is transient for all $\alpha\in(0,2]$. As we can see, the phase transition mostly occurs in dimension 1. Therefore, most of the interesting behaviors are contained in one dimensional model as we vary $\alpha\in(0,2]$. We also mention that our Proposition \ref{T113} can adapts the case $d=2$, $\alpha=2$, which might be of interest.
\end{remark}
Denote the $\sigma$-field generated by the random environment up to time $N$ by $\mathcal{G}_{N}=\sigma((\omega_{n,x})_{0\leq n\leq N, x\in\mathbbm{Z}})$. It is easy to see that the normalized partition function
\begin{equation}\label{108}
\hat{Z}_{N,\beta}^{\omega}:=\frac{Z_{N,\beta}^{\omega}}{\exp(N\lambda(\beta))}
\end{equation}
is a $\mathbbm{P}$-martingale with respect to the filtration $(\mathcal{G}_{N})_{N\geq0}$. Since $\hat{Z}_{N,\beta}^{\omega}$ is nonnegative, it converges to some random variable $\hat{Z}_{\infty,\beta}^{\omega}$ almost surely by the martingale convergence theorem. It can be seen that the event $\{\hat{Z}_{\infty,\beta}^{\omega}=0\}$ is in the tail $\sigma$-field $\bigcap\limits_{N=0}^{\infty}\sigma((\omega_{n,x})_{n\geq N, x\in\mathbbm{Z}})$. By Kolmogorov's 0-1 law, either $\mathbbm{P}(\hat{Z}_{\infty,\beta}^{\omega}=0)=0$ or $\mathbbm{P}(\hat{Z}_{\infty,\beta}^{\omega}=0)=1$. We call the first case \textit{the weak disorder regime} and the second one \textit{the strong disorder regime}. This simple but significant observation was first made by Bolthausen for a binary random environment (see \cite{bolthausen1989note}).

It is believed that in the weak disorder regime, under the polymer measure $\mathbf{P}_{N,\beta}^{\omega}$, the random walk's behavior is comparable to that under $\mathbf{P}$, i.e., the random walk fluctuates on the scale $N^{\frac{1}{\alpha}}$ (up to some extra slowly varying function) as $N\to\infty$. While in the strong disorder regime, under the polymer measure $\mathbf{P}_{N,\beta}^{\omega}$, there will be some narrow corridors at a distance $\gg N^{\frac{1}{\alpha}}$ from the origin, in which the random walk falls with high probability. In particular, the random walk's end-point distribution contains macroscopic atoms. We call this expected phenomenon in the weak disorder regime \textit{delocalization}, and the one in the strong disorder regime \textit{localization}.

One important result that connects strong disorder with localization is \cite[Theorem 2.1]{comets2003directed}, for the nearest-neighbor model which is then extended to the long-range case model in \cite{MR2381596}. We cite that result here:
\begin{theorem}[Comets~\cite{MR2381596}]\label{T103}
Denote
\begin{equation}\label{109}
I_{N}:=(\mathbf{P}_{N-1,\beta}^{\omega})^{\bigotimes2}(S_{N}^{1}=S_{N}^{2}),
\end{equation}
where $S^{1}$, $S^{2}$ are two independent copies of the random walk $S$ satisfying \eqref{101} or \eqref{102} and $(\mathbf{P}_{N-1,\beta}^{\omega})^{\bigotimes2}$ can be viewed as the distribution of the couple $(S^{1},S^{2})$ with the same environment $\omega$. Let $\beta>0$. Then,
\begin{equation}\label{110}
\mathbbm{P}(\hat{Z}_{\infty,\beta}^{\omega}=0)=\mathbbm{P}\left(\sum\limits_{N=1}^{\infty}I_{N}=\infty\right).
\end{equation}
Moreover, if $\mathbbm{P}(\hat{Z}_{\infty,\beta}^{\omega}=0)=1$, then $\mathbbm{P}$-a.s., there exist $c_{1}$, $c_{2}\in(0,\infty)$, such that
\begin{equation}\label{111}
-c_{1}\log\hat{Z}_{N,\beta}^{\omega}\leq\sum\limits_{n=1}^{N}I_{n}\leq-c_{2}\log\hat{Z}_{N,\beta}^{\omega}~~\mbox{for}~N~\mbox{large enough.}
\end{equation}
\end{theorem}
The quantity $I_{N}$ can be considered as the end-point overlap of two i.i.d.\ copies of the polymer at time $N$. For technical reasons, we consider the probability of $\{S_{N}^{1}=S_{N}^{2}\}$ under the product measure $(\mathbf{P}_{N-1,\beta}^{\omega})^{\bigotimes2}$, where the increment $S_{N}-S_{N-1}$ is distributed as the original random walk. This theorem heuristically indicates that trajectories should intersect infinitely often in strong disorder.

The free energy of the system is defined by
\begin{equation}\label{112}
F(\beta):=\lim\limits_{N\to\infty}\frac{1}{N}\log Z_{N,\beta}^{\omega}.
\end{equation}
This limit is known to exist and to be deterministic $\mathbbm{P}$-a.s. One can refer to \cite{MR2381596, comets2003directed} to see that
\begin{equation}\label{113}
F(\beta)=\lim\limits_{N\to\infty}\frac{1}{N}\mathbbm{E}[\log Z_{N,\beta}^{\omega}].
\end{equation}
We set
\begin{equation}\label{114}
p(\beta):=\lim\limits_{N\to\infty}\frac{1}{N}\log\hat{Z}_{N,\beta}^{\omega}=F(\beta)-\lambda(\beta)\leq0,
\end{equation}
where the inequality is due to Jensen's inequality. It can be seen that if $p(\beta)<0$, then $\hat{Z}_{N,\beta}^{\omega}$ has exponential decay, which implies that strong disorder holds, but not the converse. Thus, the case $p(\beta)<0$ is called \textit{the very strong disorder regime}.

Since we have the dichotomy between weak disorder and strong disorder, we can draw the phase diagram of the system, which was first established for the nearest-neighbor directed polymer in \cite{comets2006directed} and then extended to the long-range case in \cite{MR2381596}. We summarize their results as follows.
\begin{theorem}[Comets-Yoshida~\cite{comets2006directed}, Comets~\cite{MR2381596}]\label{T104}
For any $\alpha\in(0,2]$, there exist $0\leq\beta_{c}^{1}:=\beta_{c}^{1}(\alpha)\leq\beta_{c}^{2}:=\beta_{c}^{2}(\alpha)\leq\infty$, such that
\begin{equation}\label{115}
\mathbbm{P}(\hat{Z}_{\infty,\beta}^{\omega}=0)=\begin{cases}
0,&\quad\mbox{if}~\beta\in\{0\}\cup(0,\beta_{c}^{1}),\\
1,&\quad\mbox{if}~\beta>\beta_{c}^{1}.
\end{cases}
\end{equation}
And
\begin{equation}\label{116}
p(\beta)\begin{cases}
=0,&\quad\mbox{if}~\beta\in[0,\beta_{c}^{2}]\\
<0,&\quad\mbox{if}~\beta>\beta_{c}^{2}.
\end{cases}
\end{equation}
\end{theorem}
\begin{remark}\label{T105}
The reason that $p(\beta_{c}^{2})=0$ is that $p(\beta)$ is continuous in $\beta$, since $\frac{1}{N}\log\hat{Z}_{N,\beta}^{\omega}$ is convex in $\beta$. It is conjectured that there is no intermediate phase between weak disorder and very strong disorder (except at the critical point $\beta=\beta_{c}^{2}$), i.e., $\beta_{c}^{1}=\beta_{c}^{2}$. But so far this conjecture has only been proved for the nearest-neighbor directed polymer on $\mathbbm{Z}^{1+1}$ in \cite{comets2006majorizing} and on $\mathbbm{Z}^{2+1}$ in \cite{lacoin2010new}. A recent more refined result for the nearest-neighbor polymer on $\mathbbm{Z}^{2+1}$ is achieved in \cite{berger2015high}, which gives the exact asymptotic behavior of $p(\beta)$ at high temperature. Another open question is to determine whether there is weak disorder or strong disorder at critical point $\beta=\beta_{c}^{1}$.
\end{remark}

To close this subsection, we cite two quantitative results which give sufficient conditions for the existence of a weak disorder regime, respectively a strong disorder regime.
\begin{theorem}[Comets~\cite{MR2381596}]\label{T106}
If the heavy-tailed random walk $S$ is transient, and denote
\begin{equation}\label{117}
\pi_{p}:=\mathbf{P}^{\bigotimes2}(\exists n\geq1, \mbox{s.t.}~S_{n}-\tilde{S}_{n}=0)<1,
\end{equation}
where $\tilde{S}$ is an i.i.d. copy of $S$, then for all $\beta$ such that
\begin{equation}\label{118}
\lambda(2\beta)-2\lambda(\beta)<-\log{\pi_{p}},
\end{equation}
weak disorder holds.
\end{theorem}
\begin{theorem}[Comets~\cite{MR2381596}]\label{T107}
For any $\alpha\in(0,2]$, if
\begin{equation}\label{119}
\beta\lambda'(\beta)-\lambda(\beta)>-\sum\limits_{x\in\mathbbm{Z}}q(x)\log{q(x)},
\end{equation}
then $p(\beta)<0$, where $q(x):=\mathbf{P}(S_{1}=x)$.
\end{theorem}
\begin{remark}\label{T108}
By Theorem \ref{T106} and Remark \ref{T102}, there is always a weak disorder regime for $\alpha\in(0,1)$, since $\lambda(0)=0$. In Theorem \ref{T107}, $\sum\limits_{x\in\mathbbm{Z}}q(x)\log{q(x)}$ is always finite and for the random environment satisfying $\rm{essup}|\omega_{1,0}|=\infty$, we have $\beta\lambda'(\beta)-\lambda(\beta)\to\infty$ as $\beta\to\infty$ (see \cite{MR2381596}). Hence, strong disorder holds at low temperature.
\end{remark}
\subsection{Main results}
We summarize the results of this paper in this subsection. Unless otherwise specified, the random walk $S$ and the random environment $\omega$ that we consider here are introduced in subsection 1.1, from \eqref{101} to \eqref{104}. In Theorem \ref{T118}, we need some extra but mild conditions on $S$ and $\omega$, which we will mention there.

We first study the path behavior of the long-range directed polymer chain in the weak disorder regime. As in \cite[Theorem 1.2]{comets2006directed}, we will establish a stable-law version of invariance principle under the polymer measure $\mathbf{P}_{N,\beta}^{\omega}$. For heavy-tailed random walks that were introduced in \eqref{101}, \eqref{102} and \eqref{103}, define the following c$\grave{\mbox{a}}$dl$\grave{\mbox{a}}$g process
\begin{equation}\label{120}
X_{t}^{N}=\frac{S_{n}}{a_{N}},~~\mbox{for}~t\in\left[\frac{n}{N},\frac{n+1}{N}\right),~n=0,1,\ldots,N.
\end{equation}
Then $(X_{t}^{N})_{t\in[0,1]}$ converges to an $\alpha$-stable L\'evy process $(X_{t})_{t\in[0,1]}\in D[0,1]$ in distribution (see \cite[Proposition 3.4]{resnick1986point}), which we call an analogue of invariance principle for $\alpha$-stable process. Here $D[0,1]$ is the space of all functions on $[0,1]$ which are right continuous with left limits equipped with the Skorohod topology induced by the metric
\begin{equation}\label{121}
d(x,y)=\inf\limits_{\lambda\in\Lambda}\left\{\sup\limits_{0\leq t\leq1}|\lambda(t)-t|\vee\sup\limits_{0\leq t\leq1}|x(t)-y(\lambda(t))|\right\},
\end{equation}
where $\Lambda$ is the set of all the strictly increasing functions $\lambda(t)$ on $[0,1]$ with $\lambda(0)=0$ and $\lambda(1)=1$ (see \cite[Chapter 3]{billingsley2013convergence}).

Following the notations above, our first result is
\begin{theorem}\label{T109}
For the long-range directed polymer model defined in Section 1.1, assume that $\alpha\in(0,1]$ and weak disorder holds. Then for all bounded continuous functions $F$ on the path space $D[0,1]$, we have
\begin{equation}\label{122}
\mathbf{E}_{N,\beta}^{\omega}[F((X_{t}^{N})_{t\in[0,1]})]\overset{\mathbbm{P}}\to\mathbf{E}^{X}[F((X_{t})_{t\in[0,1]})]~~\mbox{as}~N\to\infty,
\end{equation}
where $\mathbf{E}^{X}$ denotes the expectation for the $\alpha$-stable L\'evy process $X$.
\end{theorem}
This theorem says that in the weak disorder regime and under the polymer measure $\mathbf{P}_{N,\beta}^{\omega}$, the polymer chain converges to the same $\alpha$-stable L\'evy process as $S$ under the measure $\mathbf{P}$. It is expected that the "in probability" convergence \eqref{122} can be improved to an "almost sure" version, but we cannot prove it for the moment.
\begin{remark}\label{T110}
In \cite{MR2381596}, Comets proved a scaling limit result for the long-range directed polymer under a stronger condition \eqref{118}, which implies weak disorder. Here by applying the procedure developed in \cite{comets2006directed}, we can weaken the condition \eqref{118} and improve the scaling limit result to the analogue of invariance principle for $\alpha$-stable process in the entire weak disorder regime.
\end{remark}

Our second result concerns the phase diagram. We can characterize the phase diagram in Theorem \ref{T104} in more detail. More precisely, we prove
\begin{theorem}\label{T111}
Following the same notations and assumptions as in Theorem \ref{T104}, we have

$(\rm{i})~\beta_{c}^{1}=0$ if and only if $S$ is recurrent.

$(\rm{ii})~\beta_{c}^{1}=\beta_{c}^{2}=0$ for $\alpha\in(1,2]$.
\end{theorem}
\begin{remark}\label{T112}
For the nearest-neighbor directed polymer model, $(\rm{i})$ has been proved in \cite[Theorem 2.3]{comets2003directed}. Our result is the analogue for long-range directed polymer.
\end{remark}
It can be seen from Theorem \ref{T106} that transience of $S$ implies the existence of a weak disorder regime. Therefore, to complete the statement of Theorem \ref{T111}\ (i), what we need to prove is the following result.
\begin{proposition}\label{T113}
If the heavy-tailed random walk $S$ is recurrent, then only strong disorder holds, i.e., $\beta_{c}^{1}=0$.
\end{proposition}
\begin{remark}\label{T114}
As we have mentioned in Remark \ref{T102}, recurrence holds for $\alpha\in(1,2]$ and transience holds for $\alpha\in(0,1)$. For the critical case $\alpha=1$, $\beta_{c}^{1}$ can be either 0 or positive, which depends on the slowly varying function $L(\cdot)$.
\end{remark}
Theorem \ref{T111}\ (ii) will be proved by showing $p(\beta)<0$ for any $\beta>0$ if $\alpha\in(1,2]$. In fact, we can give an upper bound for the free energy that we believe to be sharp up to multiplication by a constant.
\begin{theorem}\label{T115}
If $\alpha\in(1,2]$, then there exists a slowly varying function $\varphi$, which can be expressed by $\alpha$ and $L(\cdot)$, an inverse temperature $\beta_{0}>0$ and a constant $C>0$ (all depend on $\alpha$ and on $L(\cdot)$), such that for $0<\beta\leq\beta_{0}$,
\begin{equation}\label{123}
p(\beta)<-C\beta^{\frac{2\alpha}{\alpha-1}}\varphi\left(\frac{1}{\beta}\right).
\end{equation}
\end{theorem}
\begin{remark}\label{T116}
It is conjectured that the asymptotic behavior of the free energy of long-range directed polymer is $p(\beta)\sim-\textbf{F}\beta^{\frac{2\alpha}{\alpha-1}}\varphi(\frac{1}{\beta})$, where $\textbf{F}$ is the free energy of a continuum model and $\varphi$ is some function slowly varying at infinity, although the existence of $\textbf{F}$ is still an open question. For more information, see \cite[Conjecture 3.5, Conjecture 3.11]{caravenna2015polynomial}, where the authors consider the scaling limits of disordered systems, including the long-range directed polymer. Although in that paper the slowly varying function is ignored in long-range directed polymer models, it can be easily included as done in the conjecture on the critical curve of the pinning models. It is also natural to conjecture that for $\alpha=1$, $\beta_{c}^{1}=0$ can imply $\beta_{c}^{2}=0$, which we are now trying to prove. However, there is still some technical difficulty to deal with the critical case. We will shortly discuss that in Remark \ref{T307} after the proof of Theorem \ref{T115}.
\end{remark}

Our next result concerns the phenomenon of localization in the very strong disorder regime. A strong result for localization was given by Vargas, in \cite[Theorem 3.6]{vargas2007strong}, who considered $\epsilon$-atoms of polymer measure $\mathbf{P}_{N,\beta}^{\omega}$ for the nearest-neighbor directed polymer. By some modifications in the proof of his key lemma \cite[Lemma 5.3]{vargas2007strong}, we can extend his result to the long-range model.
\begin{theorem}\label{T117}
Denote
\begin{equation}\label{124}
\mathcal{A}_{N,\beta}^{\epsilon,\omega}=\{x\in\mathbbm{Z}:\mathbf{P}_{N-1,\beta}^{\omega}(S_{N}=x)>\epsilon\}.
\end{equation}
If $p(\beta)<0$, i.e., very strong disorder holds, then for $\mathbbm{P}$-a.s., there exists an $\epsilon>0$, such that
\begin{equation}\label{125}
\varliminf\limits_{N\to\infty}\frac{1}{N}\sum\limits_{n=1}^{N}\mathbf{1}_{\mathcal{A}_{n,\beta}^{\epsilon,\omega}\neq\emptyset}>0.
\end{equation}
\end{theorem}
This theorem says that in the very strong disorder regime, as the random walk moves in the random environment, there will be atoms carrying mass bigger than $\epsilon$ in the polymer's end-point distribution under $\mathbf{P}_{N,\beta}^{\omega}$ for arbitrarily large $N$.

Our last result concerns the fluctuation of the polymer in the very strong disorder regime. It has been shown in \cite{bezerra2008superdiffusivity, lacoin2011influence} that under the polymer measure, a Brownian polymer $B_{t}$ in a continuous Gaussian field will fluctuate on a scale which is not less than $t^{\frac{3}{5}}$ as $t\to\infty$, if the Gaussian field has weak coorelation. Since $t^{\frac{3}{5}}$ is larger than the underlying scale $t^{\frac{1}{2}}$, it reflects a superdiffusive phenomenon. By adapting the methods in \cite{bezerra2008superdiffusivity, lacoin2011influence}, we can establish a similar result for the long-range model. For some technical reason, we will consider a family of heavy-tailed random walks with more regular tails in a Gaussian random environment. We show that for any stable exponent $\alpha\in(1,2]$, the random walk fluctuates on a scale $\gg N^{\frac{1}{\alpha}}$ under $\mathbf{P}_{N,\beta}^{\omega}$ as $N\to\infty$. In such case, we say that the random walk has a super-$\alpha$-stable motion.
\begin{theorem}\label{T118}
Let $(X_{n})_{n\in\mathbbm{N}}$ be a sequence of i.i.d. integer-valued random variables with symmetric distribution
\begin{equation}\label{126}
\mathbf{P}(X_{1}=k)=\begin{cases}
\frac{L(|k|)}{|k|^{\alpha+1}},~&\forall k\in\mathbbm{Z}\setminus\{0\},\\
p_{0}>0,~&\mbox{for}~k=0,
\end{cases}
\end{equation}
where $L(\cdot):(0,\infty)\to(0,\infty)$ is some slowly varying function and $\alpha$ is some constant strictly larger than 1 (not necessarily less than 2). We denote the heavy-tailed random walk by
\begin{equation}\label{127}
S_{N}=\sum\limits_{n=1}^{N}X_{n}.
\end{equation}
The random environment $\omega:=(\omega_{i,x})_{(i,x)\in\mathbbm{N}\times\mathbbm{Z}}$ is a family of i.i.d. standard Gaussian random variables and we define the related polymer measure as in \eqref{105} and \eqref{106}. Then given $\alpha$ in \eqref{126} and $\beta>0$, for any arbitrarily small $\epsilon>0$, we have
\begin{equation}\label{128}
\lim\limits_{N\to\infty}\mathbbm{E}\left[\mathbf{P}_{N,\beta}^{\omega}\left(\max\limits_{1\leq n\leq N}|S_{n}|\geq\frac{\beta^{2}N}
{4(\alpha+1+\epsilon)^{2}(\log N)^{2}}\right)\right]=1.
\end{equation}
\end{theorem}
\begin{remark}\label{T119}
The condition \eqref{126} is a bit stronger than \eqref{101} or \eqref{102}, since by \cite[Proposition 1.5.8]{bingham1989regular}, \eqref{126} implies that for $\alpha\in(1,2), X_{1}$ is in the domain of attraction of the stable law with stable exponent $\alpha$, and for $\alpha\geq2$, $X_{1}$ is in the domain of attraction of the Gaussian law.
\end{remark}
In summary, in this paper, we draw a more detailed phase diagram for the long-range directed polymer model, and we extend the invariance principle in the weak disorder regime and a localization result in the very strong disorder regime from the nearest-neighbor directed polymer model to long-range directed polymer model. We also provide an upper bound for free energy of the model and a lower bound for the fluctuation scale for $\alpha\in(1,2]$. We hope that our results lay the foundation for further investigations of the long-range directed polymer model.
\subsection{Organization and strategy of the proof}
In Section 2, we will prove Theorem \ref{T109}. The procedure is the same as that in the proof of \cite[Theorem 5.1]{comets2006directed} for the nearest-neighbor model. The difference is that we need to do some estimates for heavy-tailed random walks instead of the simple random walk. We will apply some technical lemmas from \cite{comets2006directed} without stating their proof. Those lemmas can be extended to the long-range directed polymer model after careful checking.

In Section 3, we will prove Proposition \ref{T113} and Theorem \ref{T115}. For Proposition \ref{T113}, we will adapt the method used in the proof of \cite[Proposition 2.4(b)]{comets2003directed}. We will also give some equivalent criterion for the recurrence of long-range random walks. For Theorem \ref{T115}, we will use the now standard fractional moment/coarse graining/change of measure method, developed in the pinning model literature, used in \cite{lacoin2010new}.

In Section 4, we will prove Theorem \ref{T117}, which is based on the techniques developed by Vargas in \cite{vargas2007strong}.

In section 5, we will prove Theorem \ref{T118}. The methods that we will use were developed in \cite{bezerra2008superdiffusivity, lacoin2011influence}. Instead of computing the covariance of the random environment as that in \cite{bezerra2008superdiffusivity}, we will apply the change of measure method as that in \cite{lacoin2011influence}, which is also used in the proof of Theorem \ref{T115}.

Each section is independent and can be read separately.
\begin{remark}\label{T120}
One main difficulty in extending results for the nearest-neighbor model to the long-range ones is that, up to time $N$, the simple random walk can only reach at most $(2d)^{N}$ sites on $\mathbbm{Z}^{d}$, but the heavy-tailed random walk can reach infinitely many sites in one step. As we will see, the method in \cite{comets2006majorizing}, or the greedy lattice animal argument in \cite[Section 3.1]{vargas2007strong} cannot be directly applied to the long-range model.
\end{remark}
\section{Proof of Theorem \ref{T109}}
In this section, we will always assume that weak disorder holds, i.e. $\mathbbm{P}(\hat{Z}_{\infty,\beta}^{\omega}>0)=1$. According to the definition of the polymer measure $\mathbf{P}_{N,\beta}^{\omega}$, we perform change of measure for $\mathbf{P}$ at time $N$ with respect to the first $N$ steps of the random walk $S$. First we introduce the notation
\begin{equation}\label{201}
\hat{Z}_{N,\beta}^{\omega}(i,x):=\mathbf{E}^{x}\left[\exp\left(\sum\limits_{n=1}^{N}(\beta\omega_{n+i,S_{n}}-\lambda(\beta))\right)\right],
\end{equation}
where $\mathbf{E}^{x}[\cdot]$ denotes the expectation with respect to $\mathbf{P}^{x}:=\mathbf{P}(\cdot|S_{0}=x)$, the probability measure for the random walk starting at $x$. Then it is not hard to observe that given $\beta$ and $\omega$, $\mathbf{P}_{N,\beta}^{\omega}$ is an inhomogeneous Markov chain and the transition probabilities are given by
\begin{equation}\label{202}
\begin{split}
&\mathbf{P}_{N,\beta}^{\omega}(S_{i+1}=y|S_{i}=x)=\\
&\begin{cases}
&\frac{\exp(\beta\omega_{i+1,y}-\lambda(\beta))\hat{Z}_{N-i-1,\beta}^{\omega}(i+1,y)}{\hat{Z}_{N-i,\beta}^{\omega}(i,x)}
\mathbf{P}(S_{1}=y|S_{0}=x),~~\mbox{for}~0\leq i\leq N-1,\\
&\mathbf{P}(S_{1}=y|S_{0}=x),~~\mbox{for}~i\geq N.
\end{cases}
\end{split}
\end{equation}
Moreover, we can rewrite
\begin{equation}\label{203}
\hat{Z}_{N,\beta}^{\omega}(0,x)=\mathbf{E}^{x}[\exp(\beta\omega_{1,S_{1}}-\lambda(\beta))\hat{Z}_{N-1,\beta}^{\omega}(1,S_{1})].
\end{equation}
It can be seen that
\begin{equation}\label{204}
\hat{Z}_{\infty,\beta}^{\omega}(0,x):=\lim\limits_{N\to\infty}\hat{Z}_{N,\beta}^{\omega}(0,x)\geq
\mathbf{E}^{x}[\exp(\beta\omega_{1,S_{1}}-\lambda(\beta))\hat{Z}_{\infty,\beta}^{\omega}(1,S_{1})],~~\mathbbm{P}\mbox{-a.s.},
\end{equation}
where the first limit exists by martingale convergence theorem, and the inequality is due to Fatou's lemma. Notice that $(\hat{Z}_{\infty,\beta}^{\omega}(i,x))_{i\geq0, x\in\mathbbm{Z}}$ are identically distributed since $\omega$ is i.i.d., and $\omega_{1,S_{1}}$ is independent of $\hat{Z}_{\infty,\beta}^{\omega}(1,0)$. Hence, by taking expectation on both side of \eqref{204} and switch the order of $\mathbbm{E}$ and $\mathbf{E}^{x}$, we have
\begin{equation}\label{205}
\mathbbm{E}[\hat{Z}_{\infty,\beta}^{\omega}(0,x)]\geq\mathbbm{E}\left[\mathbf{E}^{x}[\exp(\beta\omega_{1,S_{1}}-\lambda(\beta))\hat{Z}_{\infty,\beta}^{\omega}(1,S_{1})
]\right]=\mathbbm{E}[\hat{Z}_{\infty,\beta}^{\omega}(1,x)].
\end{equation}
By the argument from \eqref{204} to \eqref{205}, notice that $\hat{Z}_{\infty,\beta}^{\omega}(0,x)$ and $\hat{Z}_{\infty,\beta}^{\omega}(1,x)$ have the same distribution, and then it follows that
\begin{equation}\label{206}
\hat{Z}_{\infty,\beta}^{\omega}(0,x)=\mathbf{E}^{x}[\exp(\beta\omega_{1,S_{1}}-\lambda(\beta))\hat{Z}_{\infty,\beta}^{\omega}(1,S_{1})],
~~\mathbbm{P}\mbox{-a.s.}
\end{equation}
Next, for all $A\in\mathcal{F}_{\infty}=\sigma\left(\bigcup\limits_{N=1}^{\infty}\mathcal{F}_{N}\right)$, where $\mathcal{F}_{N}$ is the $\sigma$-field generated by the first $N$ steps of the random walk $S$, the limit
\begin{equation}\label{207}
\mathbf{P}_{\infty,\beta}^{\omega}(A):=\lim\limits_{N\to\infty}\mathbf{P}_{N,\beta}^{\omega}(A)=\lim\limits_{N\to\infty}
\frac{\mathbf{E}\left[\mathbbm{1}_{A}\exp\left(\sum\limits_{n=1}^{N}\beta\omega_{n,S_{n}}-N\lambda(\beta)\right)\right]}{\hat{Z}_{N,\beta}^{\omega}}
\end{equation}
exists $\mathbbm{P}$-a.s by applying martingale convergence theorem to both the numerator and the denominator and the positivity of $\hat{Z}_{\infty,\beta}^{\omega}$.

Motivated by the argument above, we can define a random, inhomogeneous Markov chain with transition probabilities
\begin{equation}\label{208}
\mathbf{P}_{\beta,\mbox{mc}}^{\omega}(S_{i+1}=y|S_{i}=x)=\frac{\exp(\beta\omega_{i+1,y}-\lambda(\beta))\hat{Z}_{\infty,\beta}^{\omega}(i+1,y)}
{\hat{Z}_{\infty,\beta}^{\omega}(i,x)}\mathbf{P}(S_{1}=y|S_{0}=x).
\end{equation}
Note that \eqref{208} is obtained by taking limits in both numerator and denominator in \eqref{202}, which is well-defined by \eqref{206}. The reason we define $\mathbf{P}_{\beta,\mbox{mc}}^{\omega}$ is that $\mathbf{P}_{\infty,\beta}^{\omega}$ is not known to be countably additive on $\mathcal{F}_{\infty}$, while $\mathbf{P}_{\beta,\mbox{mc}}^{\omega}$ is indeed a probability measure on $\mathcal{F}_{\infty}$ and coincides with $\mathbf{P}_{\infty,\beta}^{\omega}$ on $\bigcup\limits_{n=1}^{\infty}\mathcal{F}_{n}$. The probability measure $\mathbf{P}_{\beta,\mbox{mc}}^{\omega}$ will play an important role in the proof of Theorem \ref{T109}.

We cite the following results from \cite{comets2006directed}, which we do not prove and will be used in our proof.
\subsection{Useful preliminary result}
\begin{proposition}{\cite[Proposition 4.1]{comets2006directed}}\label{T201}
Assume weak disorder.
\begin{equation}\label{209}
\mathbf{P}_{\beta,\rm{mc}}^{\omega}(A)=\mathbf{P}_{\infty}^{\beta,\omega}(A),~~\mathbbm{P}\mbox{-a.s. for all}~A\in\bigcup\limits_{N=1}^{\infty}\mathcal{F}_{N}.
\end{equation}
Moreover,
\begin{align}
&\mathbbm{E}\mathbf{P}_{\beta,\rm{mc}}^{\omega}(A)=\mathbbm{E}\mathbf{P}_{\infty,\beta}^{\omega}(A)~~\forall A\in\mathcal{F}_{\infty},\label{210}\\
&\mathbf{P}\ll\mathbbm{E}\mathbf{P}_{\beta,\rm{mc}}^{\omega}\ll\mathbf{P}~~\mbox{on}~\mathcal{F}_{\infty}.\label{211}
\end{align}
\end{proposition}
It is not hard to deduce Proposition \ref{T201} from \cite[Lemma 4.2]{comets2006directed}. We state a weaker version of \cite[Lemma 4.2]{comets2006directed} here, which will be helpful later in the proof of Proposition \ref{T205}.
\begin{lemma}{\cite[Lemma 4.2]{comets2006directed}}\label{T202}
Suppose $\{A_{N}\}_{N\geq1}\subset\mathcal{F}_{\infty}$ such that
\begin{equation}\label{212}
\lim\limits_{N\to\infty}\mathbf{P}(A_{N})=0.
\end{equation}
Then
\begin{equation}\label{213}
\lim\limits_{N\to\infty}\mathbbm{E}[\mathbf{P}_{N,\beta}^{\omega}(A_{N})]=\lim\limits_{N\to\infty}\mathbbm{E}
[\mathbf{P}_{\infty,\beta}^{\omega}(A_{N})]=0.
\end{equation}
\end{lemma}
The next proposition we cite concerns the total variation distance between the polymer measure $\mathbf{P}_{N+k,\beta}^{\omega}$ and the Markov chain $\mathbf{P}_{\beta,\mbox{mc}}^{\omega}$. We introduce the total variational norm
\begin{equation}\label{214}
||\mu-\nu||_{\mathcal{F}_{N}}:=2\sup\{\mu(A)-\nu(A):A\in\mathcal{F}_{N}\}.
\end{equation}
\begin{proposition}{\cite[Proposition 4.3]{comets2006directed}}\label{T203}
In the weak disorder regime,
\begin{equation}\label{215}
\lim\limits_{k\to\infty}\sup\limits_{N}\mathbbm{E}\left[||\mathbf{P}_{N+k,\beta}^{\omega}-\mathbf{P}_{\beta,\rm{mc}}^{\omega}||_{\mathcal{F}_{N}}\right]=0.
\end{equation}
\end{proposition}
The last result we cite here is the following lemma, which is a key ingredient to deduce our main result Theorem \ref{T207}.
\begin{lemma}{\cite[Lemma 5.3]{comets2006directed}}\label{T204}
For all $B\in\mathcal{F}_{\infty}^{\bigotimes2}$, the following limits exists $\mathbbm{P}$-a.s. in the weak disorder regime:
\begin{equation}\label{216}
(\mathbf{P}_{\infty,\beta}^{\omega})^{(2)}(B):=\lim\limits_{N\to\infty}(\mathbf{P}_{N,\beta}^{\omega})^{\bigotimes2}(B),
\end{equation}
where the definition of $(\mathbf{P}_{N-1,\beta}^{\omega})^{\bigotimes2}$ is given in Theorem \ref{T103}.

Moreover,
\begin{align}
(\mathbf{P}_{\infty,\beta}^{\omega})^{(2)}(B)&=(\mathbf{P}_{\beta,\rm{mc}}^{\omega})^{\bigotimes2}(B),~~\forall B\in\bigcup\limits_{N=1}^{\infty}\mathcal{F}_{N}^{\bigotimes2}\label{217}\\
\mathbbm{E}[(\mathbf{P}_{\infty,\beta}^{\omega})^{(2)}(B)]&=\mathbbm{E}[(\mathbf{P}_{\beta,\rm{mc}}^{\omega})^{\bigotimes2}(B)],~~\forall B\in\mathcal{F}_{\infty}^{\bigotimes2}\label{218}\\
\mathbbm{E}(\mathbf{P}_{\beta,\rm{mc}}^{\omega})^{\bigotimes2}&\ll\mathbf{P}^{\bigotimes2},~~\mbox{on}~\mathcal{F}_{\infty}^{\bigotimes2}\label{219}.
\end{align}
\end{lemma}
Note that by \cite[Remark 5.3]{comets2006directed}, we cannot identify $(\mathbf{P}_{\infty,\beta}^{\omega})^{(2)}$ with $(\mathbf{P}_{\infty,\beta}^{\omega})^{\bigotimes2}$ because we do not know whether $\mathbf{P}_{\infty,\beta}^{\omega}$ is a countably additive product measure. Although Lemma \ref{T204} looks similar to Proposition \ref{T201}, the proof of Lemma \ref{T204} is much more technical, involving Doob's decomposition of submartingale, since $\mathbf{E}^{\bigotimes2}\left[\exp\left(\sum\limits_{n=1}^{N}\beta(\omega_{n,S_{n}^{1}}+\omega_{n,S_{n}^{2}})-2N\lambda(\beta)\right)\right]$ is no longer a $\mathbbm{P}$-martingale with respect to filtration $\mathcal{G}_{N}$.
\subsection{End of the proof of Theorem \ref{T109}}
Now we can prove Theorem \ref{T109}. First, under the probability measure $\mathbf{P}_{\beta,\mbox{mc}}^{\omega}$, we establish an analogue of averaged invariance principle for the c$\grave{\mbox{a}}$dl$\grave{\mbox{a}}$g process $(X_{t}^{N})_{t\in[0,1]}$ via a second moment calculation and Proposition \ref{T201}. Since the Markov chain and the limit of the polymer measure $\mathbf{P}_{N,\beta}^{\omega}$ coincide on the $\sigma$-field generated by the random walk $S$ up to any finite time, we can apply Proposition \ref{T203} to extend the analogue of averaged invariance principle from $\mathbf{P}_{\beta,\mbox{mc}}^{\omega}$ to the polymer measure $\mathbf{P}_{N,\beta}^{\omega}$. Then, by the same procedure above, we can establish the analogue of averaged invariance principle for the i.i.d. couple $((X_{t}^{N})_{t\in[0,1]},(\tilde{X}_{t}^{N})_{t\in[0,1]})$ under the product measure $(\mathbf{P}_{\beta,\mbox{mc}}^{\omega})^{\bigotimes2}$ via Lemma \ref{T204}. Finally, since $(X_{t}^{N})_{t\in[0,1]}$ and $(\tilde{X}_{t}^{N})_{t\in[0,1]}$ are i.i.d., $\mathbf{E}_{\beta,\mbox{mc}}^{\omega}[F((X_{t}^{N})_{t\in[0,1]})]$ converges in $L^{2}$ (thus it converges in probability). The convergence in probability of $\mathbf{E}_{N,\beta}^{\omega}[F((X_{t}^{N})_{t\in[0,1]})]$ will then follow by applying Proposition \ref{T201} again.

More precisely, our first step is to establish the following proposition.
\begin{proposition}\label{T205}
Assume that $\alpha\in(0,1]$ and weak disorder holds. Then the path measures
\begin{align}
\mathbbm{E}\mathbf{P}_{N,\beta}^{\omega}((X_{t}^{N})_{t\in[0,1]}\in\cdot)&\Rightarrow\mathbf{P}^{X}((X_{t})_{t\in[0,1]}\in\cdot)~~\mbox{weakly as}~N\to\infty,
\label{220}\\
\mathbbm{E}\mathbf{P}_{\beta,\rm{mc}}^{\omega}((X_{t}^{N})_{t\in[0,1]}\in\cdot)&\Rightarrow\mathbf{P}^{X}((X_{t})_{t\in[0,1]}\in\cdot)~~\mbox{weakly as}~N
\to\infty.
\label{221}
\end{align}
\end{proposition}
\begin{remark}\label{T206}
The analogue of Proposition \ref{T205} was proved for the nearest-neighbor model in \cite{comets2006directed}. To extend it to the long-range model, we use the observation that under the Skorohod distance, $X_{t}^{N}$ and $X_{t}^{N-k}$ are close for fixed $k$ and large enough $N$.
\end{remark}
Applying Proposition \ref{T205}, we will then prove
\begin{theorem}\label{T207}
Assume that $\alpha\in(0,1]$ and weak disorder holds. Then, for all bounded continuous functions $F$ on the path space $D[0,1]$,
\begin{align}
\mathbf{E}_{N,\beta}^{\omega}[F((X_{t}^{N})_{t\in[0,1]})]&\overset{\mathbbm{P}}\to\mathbf{E}^{X}[F((X_{t})_{t\in[0,1]})]~~\mbox{as}~N\to\infty.\label{222}\\
\mathbf{E}_{\beta,\rm{mc}}^{\omega}[F((X_{t}^{N})_{t\in[0,1]})]&\overset{\mathbbm{P}}\to\mathbf{E}^{X}[F((X_{t})_{t\in[0,1]})]~~\mbox{as}~N\to\infty.\label{223}
\end{align}
\end{theorem}
\begin{proof}[Proof of Proposition \ref{T205}]
We first prove \eqref{221}. Since the path space $D[0,1]$ is separable, by \cite[Theorem 11.3.3]{dudley2002real}, it suffices to show that
\begin{equation}\label{224}
\lim\limits_{N\to\infty}\mathbbm{E}[\mathbf{E}_{\beta,\mbox{mc}}^{\omega}[F((X_{t}^{N})_{t\in[0,1]})]=\mathbf{E}^{X}[F((X_{t})_{t\in[0,1]})],~~\forall F\in\mbox{BL}(D[0,1]),
\end{equation}
where BL$(D[0,1])$ is the set of all the bounded Lipschitz functionals on $D[0,1]$. To simplify the notations, we denote $F((X_{t}^{N})_{t\in[0,1]})$ by $f_{N}$ and $F((X_{t})_{t\in[0,1]})$ by $f$.

Our first statement is that for any sequence $(N_{k})_{k\geq1}$, such that for all $k\geq1$, $\frac{N_{k+1}}{N_{k}}\geq\rho>1$, we have
\begin{equation}\label{225}
\frac{1}{n}\sum\limits_{k=1}^{n}f_{N_{k}}\overset{\mathbf{P}}\to\mathbf{E}^{X}[f],~~\mbox{as}~n\to\infty.
\end{equation}
To prove \eqref{225}, we start by observing that
\begin{equation}\label{226}
\begin{split}
&\mathbf{P}\left(\left|\frac{1}{n}\sum\limits_{k=1}^{n}f_{N_{k}}-\mathbf{E}^{X}[f]\right|>\epsilon\right)\\
\leq&\mathbf{P}\left(\left|\frac{1}{n}\sum\limits_{k=1}^{n}\left(f_{N_{k}}-\mathbf{E}[f_{N_{k}}]\right)\right|>
\frac{\epsilon}{2}\right)
+\mathbf{P}\left(\left|\frac{1}{n}\sum\limits_{k=1}^{n}\mathbf{E}[f_{N_{k}}]-\mathbf{E}^{X}[f]\right|>\frac{\epsilon}
{2}\right).
\end{split}
\end{equation}
The second term on the right-hand side vanishes as $n$ tends to infinity by the analogue of invariance principle for stable laws. For the first term,
\begin{equation}\label{227}
\begin{split}
&\mathbf{P}\left(\left|\frac{1}{n}\sum\limits_{k=1}^{n}\left(f_{N_{k}}-\mathbf{E}[f_{N_{k}}]\right)\right|>\frac{\epsilon}{2}\right)
\leq\frac{4}{n^{2}\epsilon^{2}}\mathbf{E}\left|\sum\limits_{k=1}^{n}\left(f_{N_{k}}-\mathbf{E}[f_{N_{k}}]
\right)\right|^{2}\\
\leq&\frac{4}{n^{2}\epsilon^{2}}\sum\limits_{k=1}^{n}\mathbf{E}\left(f_{N_{k}}-\mathbf{E}[f_{N_{k}}]\right)^{2}
+\frac{8}{n^{2}\epsilon^{2}}\sum\limits_{k=1}^{n}\sum\limits_{j=k+1}^{n}\left|\mathbf{E}\left[\left(f_{N_{k}}-\mathbf{E}[f_{N_{k}}]\right)
\left(f_{N_{j}}-\mathbf{E}[f_{N_{j}}]\right)\right]\right|.
\end{split}
\end{equation}
The first term on right hand side is bounded by $\mathcal{O}(\frac{1}{n})$, since $F$ is bounded. For the second term on the right hand side, by the method that was used in \cite[Page 99]{atlagh2000theoreme}, each term in the summation is bounded by $C(p)\left(\frac{a_{N_{k}}}{a_{N_{j}}}\right)^{p}$ for any $p<\alpha$ and then bounded by
$C(\delta)\rho^{-(\frac{1}{\alpha}-\delta)(j-k)}$ further for some $0<\delta<\frac{1}{\alpha}$ (see Potter bounds in \cite{bingham1989regular}), where $C(p)$ and $C(\delta)$ are constants only depending on $p$ and $\delta$ respectively (one can find the full details in \cite[Theorem 4.18]{jonsson2007almost}). Therefore, the summation in the second term is also bounded by $\mathcal{O}(\frac{1}{n})$. Combine \eqref{226} and \eqref{227}, we obtain \eqref{225}. By \eqref{211}, the convergence in \eqref{225} also holds in $\mathbbm{E}\mathbf{P}_{\beta,\mbox{mc}}^{\omega}$-probability.

Denote $E_{N}=\mathbbm{E}[\mathbf{E}_{\beta,\mbox{mc}}^{\omega}[f_{N}]]$. For any converging subsequence $E_{N_{k}}$, we can find a sub-subsequence $E_{N_{k_{j}}}$, such that $\inf\limits_{j}(N_{k_{j+1}}/N_{k_{j}})=\rho>1$, and then by \eqref{225} and bounded convergence theorem, $\lim\limits_{n\to\infty}\frac{1}{n}\sum\limits_{j=1}^{n}E_{N_{k_{j}}}=\mathbf{E}^{X}[f]$. Therefore we conclude that \eqref{224} holds.

Next we prove \eqref{220}. The basic idea is the same as the proof of \eqref{221}, we only need to prove that for all $F\in\mbox{BL}(D[0,1])$,
\begin{equation}\label{228}
\lim\limits_{N\to\infty}\mathbbm{E}[\mathbf{E}_{N,\beta}^{\omega}[F((X_{t}^{N})_{t\in[0,1]})]]=\mathbf{E}^{X}[F((X_{t})_{t\in[0,1]})].
\end{equation}
For $0\leq k\leq N$,
\begin{equation}\label{229}
\begin{split}
&\left|\mathbbm{E}\left[\mathbf{E}_{N,\beta}^{\omega}\left[f_{N}-\mathbf{E}^{X}[f]\right]\right]\right|\\
\leq&\mathbbm{E}\left[\mathbf{E}_{N,\beta}^{\omega}\left|f_{N}-f_{N-k}\right|\right]\\
&+\mathbbm{E}\left|\mathbf{E}_{N,\beta}^{\omega}[f_{N-k}]-\mathbf{E}_{\beta,\mbox{mc}}^{\omega}[f_{N-k}]\right|\\
&+\left|\mathbbm{E}\left[\mathbf{E}_{\beta,\mbox{mc}}^{\omega}[f_{N-k}]\right]-\mathbf{E}^{X}[f]\right|.
\end{split}
\end{equation}
For any fixed $k$, let $N$ tend to infinity, then by \eqref{224}, the last term vanishes. For the first term, denote $d((X_{t}^{N})_{t\in[0,1]},
(X_{t}^{N-k})_{t\in[0,1]})$ by $d(N,k)$, where $d(\cdot,\cdot)$ is the Skorohod metric on $D[0,1]$, which was introduced in \eqref{121}. Then for any $\delta>0$, we have
\begin{equation}\label{230}
\mathbbm{E}\left[\mathbf{E}_{N,\beta}^{\omega}\left|f_{N}-f_{N-k}\right|\right]\leq L\mathbbm{E}[\mathbf{E}_{N,\beta}^{\omega}[d(N,k)\mathbbm{1}_{d(N,k)\leq\delta}]]+2\left(\sup\limits_{x\in D[0,1]}|F(x)|\right)\mathbbm{E}
[\mathbf{E}_{N,\beta}^{\omega}[\mathbbm{1}_{d(N,k)>\delta}]],
\end{equation}
where $L$ is the Lipschitz norm of $F$. The first term on the right-hand side of \eqref{230} can be made sufficiently small by choosing $\delta$ sufficiently small. The expectation in the second term can be bounded by
\begin{equation}\label{231}
\mathbbm{E}\left[\mathbf{P}_{N,\beta}^{\omega}\left(\left\{\sup\limits_{1\leq j\leq N-k}\left|\frac{S_{j}}{a_{N}}-\frac{S_{j}}{a_{N-k}}\right|>\delta\right\}\bigcup\left\{\sup\limits_{1\leq j\leq k}\left|\frac{S_{N-k+j}}{a_{N}}-\frac{S_{N-k}}{a_{N-k}}\right|>\delta\right\}\right)\right],
\end{equation}
since Skorokhod distance allows us to align the jumps of two different c$\grave{\mbox{a}}$dl$\grave{\mbox{a}}$g functions. To be specific here, in \eqref{121}, we can choose $\lambda(t)=\frac{N-k}{N}$ on $\left[0,\frac{N-k-1}{N}\right]$ and linear on $\left[\frac{N-k-1}{N},1\right]$ to make the first $N-k-1$ jumps of both $X_{t}^{N}$ and $X_{t}^{N-k}$ occur at the same time, which gives an upper bound \eqref{231}. We observe that
\begin{equation}\label{232}
\begin{split}
&\mathbf{P}\left(\left\{\sup\limits_{1\leq j\leq N-k}\left|\frac{S_{j}}{a_{N}}-\frac{S_{j}}{a_{N-k}}\right|>\delta\right\}\bigcup\left\{\sup\limits_{1\leq j\leq k}\left|\frac{S_{N-k+j}}{a_{N}}-\frac{S_{N-k}}{a_{N-k}}\right|>\delta\right\}\right)\\
\leq&\mathbf{P}\left(\sup\limits_{1\leq j\leq N-k}\left|\frac{S_{j}}{a_{N-k}}\right|>\frac{\delta}{|1-\frac{a_{N-k}}{a_{N}}|}\right)
+\mathbf{P}\left(\left|\frac{S_{N-k}}{a_{N-k}}\right|>\frac{\delta}{2|1-\frac{a_{N-k}}{a_{N}}|}\right)\\
&+\mathbf{P}\left(\sup\limits_{1\leq j\leq k}\left|\frac{S_{N-k+j}-S_{N-k}}{a_{N}}\right|>\frac{\delta}{2}\right).
\end{split}
\end{equation}
Note that $\frac{a_{N-k}}{a_{N}}\to1$, as $N\to\infty$. By weak convergence of $a_{N-k}^{-1}S_{N-k}$, the continuous mapping theorem and the fact that $\sup\limits_{0\leq t\leq 1}|X_{t}|<\infty$ a.s., the first two terms on the right-hand side of \eqref{232} tend to $0$ as $N$ tends to infinity. The last term also tends to $0$, since $S_{N-k+j}-S_{N-k}\overset{d}{=}S_{j}$. Denote
\begin{equation}\label{233}
\left\{\sup\limits_{1\leq j\leq N-k}\left|\frac{S_{j}}{a_{N}}-\frac{S_{j}}{a_{N-k}}\right|>\delta\right\}\bigcup\left\{\sup\limits_{1\leq j\leq k}\left|\frac{S_{N-k+j}}{a_{N}}-\frac{S_{N-k}}{a_{N-k}}\right|>\delta\right\}
\end{equation}
by $A_{N,k}$ for $N>k$. We have $\lim\limits_{N\to\infty}\mathbf{P}(A_{N,k})=0$. Then by Lemma \ref{T202}, \eqref{231} tends to $0$ as $N$ tends to infinity. Therefore, the first term on the right-hand side of \eqref{229} vanishes.

Finally, the second term on the right-hand side of \eqref{229} is bounded by\\ $\sup\limits_{N}\mathbbm{E}\left|\mathbf{E}_{N+k,\beta}^{\omega}[f_{N}]-\mathbf{E}_{\beta,\mbox{mc}}^{\omega}[f_{N}]\right|$. Note that $f_{N}$ is measurable w.r.t. $\mathcal{F}_{N}$. Hence
\begin{equation}\label{234}
\sup\limits_{N}\mathbbm{E}\left|\mathbf{E}_{N+k,\beta}^{\omega}[f_{N}]-\mathbf{E}_{\beta,\mbox{mc}}^{\omega}[f_{N}]\right|
\leq\left(\sup\limits_{x\in D[0,1]}|F(x)|\right)\sup\limits_{N}\mathbbm{E}\left[||\mathbf{P}_{N+k,\beta}^{\omega}-\mathbf{P}_{\beta,\mbox{mc}}^{\omega}||
_{\mathcal{F}_{N}}\right].
\end{equation}
Let $k$ tend to infinity and apply Proposition \ref{T203}. The right-hand side of \eqref{234} tends to $0$. This completes the proof of \eqref{220}.
\end{proof}
\begin{proof}[Proof of Theorem \ref{T207}]
By the same procedure as in the proof of \eqref{221}, but using Lemma \ref{T204} instead of Proposition \ref{T201}, for any $G\in\mbox{C}_{b}(D[0,1]\times D[0,1])$, we have
\begin{equation}\label{235}
\lim\limits_{N\to\infty}\mathbbm{E}\left[(\mathbf{E}_{\beta,\rm{mc}}^{\omega})^{\bigotimes2}[G((X_{t}^{N})_{t\in[0,1]},(\tilde{X}_{t}^{N})_{t\in[0,1]})]\right]
=(\mathbf{E}^{X})^{\bigotimes2}[G((X_{t})_{t\in[0,1]},(\tilde{X}_{t})_{t\in[0,1]})].
\end{equation}
If we choose $G(x,\tilde{x})=(F(x)-\mathbf{E}^{X}[f])(F(\tilde{x})-\mathbf{E}^{X}[f])$, then it follows that
\begin{equation}\label{236}
\lim\limits_{N\to\infty}\mathbbm{E}\left[\left(\mathbf{E}_{\beta,\rm{mc}}^{\omega}[f_{N}-\mathbf{E}^{X}[f]]\right)^{2}\right]=0,
\end{equation}
which proves \eqref{223}. To prove (\ref{218}), it suffices to show that for all $F\in\mbox{C}_{b}(D[0,1])$,
\begin{equation}\label{237}
\lim\limits_{N\to\infty}\mathbbm{E}\left|\mathbf{E}_{N,\beta}^{\omega}[f_{N}-\mathbf{E}^{X}[f]]\right|=0.
\end{equation}
The proof of \eqref{237} is the same as that of \eqref{228}.
\end{proof}
\section{Proof of Proposition \ref{T113} and Theorem \ref{T115}}
\subsection{Proof of Proposition \ref{T113}}
We will first give some equivalent conditions for the recurrence of heavy-tailed random walks, which will be used later.
\begin{proposition}\label{T301}
Suppose that $S=(S_{n})_{n\leq0}$ is a heavy-tailed random walk satisfying \eqref{103} with $b_{n}=0$.

$(\rm{i})~S$ is recurrent if and only if $\sum\limits_{n=1}^{\infty}\frac{1}{a_{n}}=\infty$, where $a_{n}=n^{\frac{1}{\alpha}}l(n)$ for some slowly varying function $l(n)$, is the scaling factor in \eqref{103}.

$(\rm{ii})~$If $\mathbf{P}$ is in the domain of attraction of the Cauchy distribution, i.e., $\alpha=1$, then $S$ is recurrent if and only if $\sum\limits_{n=1}^{\infty}\frac{1}{nL(n)}=\infty$, where $L(n)$ is some slowly varying function defined in $\eqref{101}$.
\end{proposition}
\begin{proof}
(i)~For $\alpha\in(0,1)$, the random walk $S$ is always transient (see Remark \ref{T102}), and $\sum\limits_{n=1}^{\infty}a_{n}=\sum\limits_{n=1}^{\infty}\frac{1}{n^{\frac{1}{\alpha}}l(n)}<\infty$ since $\frac{1}{\alpha}>1$. Hence, the result is obvious.

For $\alpha\in[1,2]$, $S_{1}$ should take both positive and negative values. We can set the possible smallest return time $k$ by the greatest common divisor of $\{n\in\mathbbm{N}:\mathbf{P}(S_{n}=0)>0\}$, which is finite. By Gnedenko's local limit theorem (see \cite[Theorem 8.4.1]{bingham1989regular}), we have
\begin{equation}\label{301}
\mathbf{P}(S_{nk}=0)\sim\frac{g_{\alpha}(0)h}{a_{nk}}~\mbox{as}~n\to\infty,
\end{equation}
where $h$ is the largest integer such that $\{z+h\mathbbm{Z}\}$ contains all the values of $S_{1}$ for some integer $z$ and $g_{\alpha}$ is the density function of the limiting stable distribution $X_{\alpha}$. Note that $S$ is recurrent if and only if $\sum\limits_{n=0}^{\infty}\mathbf{P}(S_{nk}=0)=\sum\limits_{n=0}^{\infty}\mathbf{P}(S_{n}=0)=\infty$, and $\sum\limits_{m=0}^{k-1}\frac{1}{a_{(n-1)k+m}}$ has the same order of $\frac{k}{a_{nk}}$ as $n\to\infty$ by Uniform Convergence Theorem of slowly varying function (\cite[Theorem 1.2.1]{bingham1989regular}), then the result follows by \eqref{301}.

(ii) Again, by \cite[Theorem 1.2.1]{bingham1989regular}, for any slowly varying function $L(x)$, there exist two constants $C_{1}$ and $C_{2}$, such that $C_{1}<\frac{L(x)}{L(n)}<C_{2}$, for $x\in[n, n+1)$. Hence, $\sum\limits_{n=1}^{\infty}\frac{1}{nL(n)}=\infty\Leftrightarrow\int_{1}^{\infty}\frac{dt}{tL(t)}=\infty$.

By \cite[Proposition 1.3.4]{bingham1989regular}, we can extend $a_{n}$ to a regularly varying function $a(t)=tl(t)$ for $t\in(0,\infty)$ and further assume that $a(t)$ is non-decreasing and differentiable. By \cite[Proposition 1.5.8]{bingham1989regular}, $\frac{d}{dt}a(t)\sim\frac{a(t)}{t}$. Note that $a_{n}\sim nL(a_{n})$ since $a_{n}$ can be chosen by $n\mathbf{P}(|S_{1}|>n)\sim1$ (see \cite[Chapter 7]{newell1955limit}), we then obtain
\begin{equation}\label{302}
\int_{1}^{\infty}\frac{dt}{a(t)}=\infty\Leftrightarrow\int_{1}^{\infty}\frac{dt}{tL(a(t))}=\infty\Leftrightarrow\int_{1}^{\infty}
\frac{ds}{sL(s)}=\infty,
\end{equation}
where the last equivalence follows from the change of variables $s=a(t)$. Now the result holds by part (i).
\end{proof}
\begin{remark}\label{T302}
By \cite[Theorem 8.3.4]{chung2001course}, a random walk $S$ whose expectation $\mathbf{E}[S_{1}]$ exists is recurrent if and only if $\mathbf{E}[S_{1}]=0$. For $\alpha\in(1,2]$, since $S_{1}-\mathbf{E}[S_{1}]$ has expectation 0 and $\sum\limits_{n=1}^{\infty}\frac{1}{a_{n}}=\infty$ holds always, hence, setting $b_{n}=0$ does not reduce much generality.
\end{remark}
To prove Proposition \ref{T113}, we apply the fractional moment method as in the proof of \cite[Theorem 2.3(b)]{comets2003directed}. We cite two lemmas \cite[Lemma 3.1, Lemma 4.2]{comets2003directed} here without proof.
\begin{lemma}\label{T303}
Let $(\xi_{i})_{i\geq1}$ be positive, non-constant i.i.d. random variables such that $\mathbbm{E}[\xi_{1}]=1$ and $\mathbbm{E}[\xi_{1}^{3}+\log^{2}\xi_{1}]<\infty$. For $(\alpha_{i})_{i\geq1}\in[0,1]^{\mathbbm{N}}$ such that $\sum\limits_{i=1}^{\infty}\alpha_{i}=1$, define a centered random variable $U>-1$ by $U=\sum\limits_{i=1}^{\infty}\alpha_{i}\xi_{i}-1$. Then there exists a constant $c\in(0,\infty)$, independent of $(\alpha_{i})_{i\geq1}$, such that
\begin{equation}\label{303}
\frac{1}{c}\sum\limits_{i=1}^{\infty}\alpha_{i}^{2}\leq\mathbbm{E}\left[\frac{U^{2}}{2+U}\right].
\end{equation}
\end{lemma}
\begin{remark}\label{T304}
In \cite{comets2003directed}, the authors considered sequences $(\alpha_{i})_{1\leq i\leq n}$ for any finite $n$. It can be seen that the proof for a countable sequence $(\alpha_{i})_{i\geq1}$ follows the same lines as that for finite $(\alpha_{i})_{1\leq i\leq n}$. Note that $U$ is a well defined random variable by monotone convergence theorem.
\end{remark}
\begin{lemma}\label{T305}
Recall the overlap $I_{N}$ from \eqref{109}. For $\theta\in[0,1]$ and $\Lambda\in\mathbbm{Z}$,
\begin{equation}\label{304}
\mathbbm{E}[(\hat{Z}_{N-1,\beta}^{\omega})^{\theta}I_{N}]\geq\frac{1}{|\Lambda|}\mathbbm{E}[(\hat{Z}_{N-1,\beta}^{\omega})^{\theta}]-\frac{2}{|\Lambda|}
\mathbf{P}(S_{N}\notin\Lambda)^{\theta}.
\end{equation}
\end{lemma}
\begin{proof}[Proof of Proposition \ref{T113}]
We will show $\lim\limits_{N\to\infty}\mathbbm{E}[(\hat{Z}_{N,\beta}^{\omega})^{\theta}]=0$ for some $\theta\in(0,1)$ via a recursive inequality between $\mathbbm{E}[(\hat{Z}_{N,\beta}^{\omega})^{\theta}]$ and $\mathbbm{E}[(\hat{Z}_{N-1,\beta}^{\omega})^{\theta}]$.

We first establish the connection between $\hat{Z}_{N,\beta}^{\omega}$ and $\hat{Z}_{N-1,\beta}^{\omega}$ by writing
\begin{equation}\label{305}
\frac{\hat{Z}_{N,\beta}^{\omega}}{\hat{Z}_{N-1,\beta}^{\omega}}=U_{N,\beta}^{\omega}+1,
\end{equation}
where it can be seen that
\begin{equation}\label{306}
U_{N,\beta}^{\omega}=\mathbf{E}_{N-1,\beta}^{\omega}[\exp(\beta\omega_{N,S_{N}}-\lambda(\beta))]-1.
\end{equation}
Therefore, conditionally on $\mathcal{G}_{N-1}$, which is the $\sigma$-field generated by $(\omega_{i,x})_{0\leq i\leq N-1, x\in\mathbbm{Z}}$, $U_{N}$ satisfies the definition of $U$ in Lemma \ref{T303}. Then we have
\begin{equation}\label{307}
\mathbbm{E}[(\hat{Z}_{N,\beta}^{\omega})^{\theta}|\mathcal{G}_{N-1}]=(\hat{Z}_{N-1,\beta}^{\omega})^{\theta}\mathbbm{E}
[(U_{N,\beta}^{\omega}+1)^{\theta}|\mathcal{G}_{N-1}].
\end{equation}
To deal with the right-hand side of \eqref{307}, we define an auxiliary function. Assume $\theta\in(0,1)$. Set $f:(-1,\infty)\mapsto[0,\infty)$ by
\begin{equation}\label{308}
f(u)=1+\theta u-(1+u)^{\theta}.
\end{equation}
It is easy to see that there exist $c_{1},c_{2}\in(0,\infty)$ such that for all $u\in(-1,\infty)$, we have
\begin{equation}\label{309}
\frac{c_{1}u^{2}}{2+u}\leq f(u)\leq c_{2}u^{2}.
\end{equation}
Notice that the left-hand side of \eqref{309} has the form of the right hand side of \eqref{303}.

Then
\begin{equation}\label{310}
\begin{split}
&(\hat{Z}_{N-1,\beta}^{\omega})^{\theta}\mathbbm{E}
[(U_{N,\beta}^{\omega}+1)^{\theta}|\mathcal{G}_{N-1}]\\
=&(\hat{Z}_{N-1,\beta}^{\omega})^{\theta}\mathbbm{E}[1+\theta U_{N,\beta}^{\omega}-f(U_{N,\beta}^{\omega})|\mathcal{G}_{N-1}]\\
=&(\hat{Z}_{N-1,\beta}^{\omega})^{\theta}-(\hat{Z}_{N-1,\beta}^{\omega})^{\theta}\mathbbm{E}[f(U_{N,\beta}^{\omega})|\mathcal{G}_{N-1}]\\
\leq&(\hat{Z}_{N-1,\beta}^{\omega})^{\theta}-(\hat{Z}_{N-1,\beta}^{\omega})^{\theta}\mathbbm{E}\left[\frac{c_{1}(U_{N,\beta}^{\omega})^{2}}
{2+U_{N,\beta}^{\omega}}\big|\mathcal{G}_{N-1}\right]\\
\leq&(\hat{Z}_{N-1,\beta}^{\omega})^{\theta}-c_{3}(\hat{Z}_{N-1,\beta}^{\omega})^{\theta}I_{N},
\end{split}
\end{equation}
where the last inequality is due to Lemma \ref{T303}, with $(\alpha_{x})_{x\in\mathbbm{Z}}=(\mathbf{P}_{N-1,\beta}^{\omega}(S_{N}=x))_{x\in\mathbbm{Z}}$ and noticing that $I_{N}=\sum\limits_{x\in\mathbbm{Z}}(\mathbf{P}_{N-1,\beta}^{\omega}(S_{N}=x))^{2}$. Taking expectation on both sides of \eqref{307} and \eqref{310} and using Lemma \ref{T305}, we obtain
\begin{equation}\label{311}
\mathbbm{E}[(\hat{Z}_{N,\beta}^{\omega})^{\theta}]\leq\left(1-\frac{c_{3}}{|\Lambda_{N}|}\right)\mathbbm{E}[(\hat{Z}_{N-1,\beta}^{\omega})^{\theta}]
+\frac{2c_{3}}{|\Lambda_{N}|}\mathbf{P}(S_{N}\notin\Lambda_{N})^{\theta}
\end{equation}
for any sequence of bounded sets $(\Lambda_{i})_{i\geq1}$.

For a recurrent $S$, by Proposition \ref{T301} (i), we have $\sum\limits_{n=1}^{\infty}\frac{1}{a_{n}}=\infty$. Then we can always find a sequence $(b_{n})_{n\geq1}$ such that
\begin{equation}\label{312}
\lim\limits_{n\to\infty}\frac{b_{n}}{a_{n}}=\infty,~\mbox{and}~\sum\limits_{n=1}^{\infty}\frac{1}{b_{n}}=\infty.
\end{equation}
Hence, we can choose $\Lambda_{N}=(-b_{N},b_{N})$ such that $\mathbf{P}(S_{N}\notin\Lambda_{N})$ tends to 0, since $a_{N}^{-1}S_{N}$ converges in probability to some stable law. For any $\epsilon>0$, for large enough $N$, we have $2\mathbf{P}(S_{N}\notin\Lambda_{N})^{\theta}<\epsilon$, and then
\begin{equation}\label{313}
\mathbbm{E}[(\hat{Z}_{N,\beta}^{\omega})^{\theta}]-\epsilon\leq\left(1-\frac{c_{3}}{2b_{N}}\right)
(\mathbbm{E}[(\hat{Z}_{N-1,\beta}^{\omega})^{\theta}]-\epsilon)\leq\exp\left(-\frac{c_{3}}{2b_{N}}\right)(\mathbbm{E}[(\hat{Z}_{N-1,\beta}^{\omega})^{\theta}]
-\epsilon).
\end{equation}
Iterating this inequality and using Fatou's lemma, we obtain
\begin{equation}\label{314}
\mathbbm{E}[(\hat{Z}_{\infty,\beta}^{\omega})^{\theta}]-\epsilon\leq\varliminf\limits_{M\to\infty}\mathbbm{E}[(\hat{Z}_{M,\beta}^{\omega})^{\theta}]-\epsilon
\leq\varliminf\limits_{M\to\infty}\exp\left(-\sum\limits_{n=N}^{M}\frac{c_{3}}{2b_{n}}\right)
(\mathbbm{E}[(\hat{Z}_{N-1,\beta}^{\omega})^{\theta}]-\epsilon)=0.
\end{equation}
Since $\epsilon$ is arbitrary, it follows that $\mathbbm{E}[(\hat{Z}_{\infty,\beta}^{\omega})^{\theta}]=0$, i.e., strong disorder holds.
\end{proof}

\subsection{Proof of Theorem \ref{T115}}
In this subsection, we prove Theorem \ref{T115}, which gives bounds on the free energy when $\alpha\in(1,2]$. The technique that is used here has been developed in many articles, see \cite{lacoin2010new, toninelli2009coarse, giacomin2010marginal}. We only give a proof for Gaussian environment. It is not hard to deduce the result for general environment from Gaussian environment, see \cite[Page 481]{lacoin2010new}.
\begin{proof}[Proof of Theorem \ref{T115} in Gaussian environment]
We start with a simple observation. By Jensen's inequality, for $\theta\in(0,1)$,
\begin{equation}\label{315}
p(\beta)=\lim\limits_{N\to\infty}\frac{1}{N}\mathbbm{E}[\log\hat{Z}_{N,\beta}^{\omega}]\leq\varliminf\limits_{N\to\infty}\frac{1}{\theta N}\log\mathbbm{E}
[(\hat{Z}_{N,\beta}^{\omega})^{\theta}].
\end{equation}
Hence, we only need to show that the fractional moment of $\hat{Z}_{N,\beta}^{\omega}$, for some power $\theta\in(0,1)$ that will be determined later, decays exponentially in $N$.

To conclude \eqref{123}, it is sufficient to focus on a subsequence of $\hat{Z}_{N,\beta}^{\omega}$ by \eqref{315}. We use the coarse-graining method in this step. Consider the sequence $N=mn$, where $m$ will tend to infinity and $n$ is fixed once chosen, which will be determined by $\beta$ later. The idea is that we will only investigate the heavy-tailed random walk $S$ at time $n, 2n,\ldots, mn$. For each $in$, where $i=1,\ldots,m$, we can find a time-space window in $\mathbbm{N}\times\mathbbm{Z}$, in which $S_{in}$ falls with high probability, thanks to convergence to stable law.

Let $(a_{n})_{n\geq1}$ be the scaling sequence such that $a_{N}^{-1}S_{N}$ converges to an $\alpha$-stable law in distribution. Notice that we can choose $(a_{n})_{n\geq1}$ to be non-decreasing and integer-valued, which will simplify our argument. Denote $I_{k}=[ka_{n},(k+1)a_{n})$ and we make the decomposition
\begin{equation}\label{316}
\hat{Z}_{N,\beta}^{\omega}=\sum\limits_{y_{1},\ldots,y_{m}\in\mathbbm{Z}}\hat{Z}_{(y_{1},\ldots,y_{m})}^{\beta,\omega},
\end{equation}
where
\begin{equation}\label{317}
\hat{Z}_{(y_{1},\ldots,y_{m})}^{\beta,\omega}=\mathbf{E}\left[\exp\left\{\sum\limits_{i=1}^{N}\left(\beta\omega_{i,S_{i}}-\frac{\beta^{2}}{2}\right)\right\}
\mathbbm{1}_{\{S_{in}\in I_{y_{i}}, \forall i=1,\ldots,m\}}\right].
\end{equation}
Then
\begin{equation}\label{318}
\mathbbm{E}[(\hat{Z}_{N,\beta}^{\omega})^{\theta}]\leq\sum\limits_{y_{1},\ldots,y_{m}\in\mathbbm{Z}}\mathbbm{E}[(\hat{Z}_{(y_{1},\ldots,y_{m})}^{\beta,\omega})
^{\theta}],
\end{equation}
since the inequality $(\sum a_{n})^{\theta}\leq\sum a_{n}^{\theta}$ holds for any countable sequence for any $\theta\in(0,1]$. Note that the length of each interval $I_{k}$ is chosen to match the scaling of $S_{n}$, and if $S\in\{S_{in}\in I_{y_{i}}~\forall i=1,\ldots,m\}$, then $(y_{1},\ldots,y_{m})$ is called the coarse-grained version of the trajectory of $S$.

Next, to estimate $\mathbbm{E}[(\hat{Z}_{(y_{1},\ldots,y_{m})}^{\beta,\omega})^{\theta}]$, we use a change of measure procedure, which we now explain. We will define a new law for the random environment, which shifts down the expectation of $\omega_{j,x}$ at sites, where the random walk $S$ visits with relatively high probability, to a negative value. This can significantly decrease the expectation of $\hat{Z}_{(y_{1},\ldots,y_{m})}^{\beta,\omega}$ under the new law of $\omega$, and the cost of the change of measure can be chosen to be small.

For any $Y=(y_{0},\ldots,y_{m-1})$, we introduce the set
\begin{equation}\label{319}
J_{Y}=\{(kn+i,y_{k}a_{n}+z):~k=0,\ldots,m-1,~i=1,\ldots,n,~|z|\leq C_{1}a_{n}\},
\end{equation}
where $y_{0}=0$ for convenience and $C_{1}$ is a large integer to be determined later. Note that $|J_{Y}|=2C_{1}a_{n}mn$, where $|A|$ denote the cardinality of a set $A$. We can consider the choice of $J_{Y}$ in the following way: suppose that the random walk $S$ reaches $y_{k}a_{n}$ at time $kn$. Then for the next $n$ steps of this random walk, its path will probably fall in the set
\begin{equation}\label{320}
B_{k}=\{(kn+i,y_{k}a_{n}+z):~i=1,\ldots,n,~|z|\leq C_{1}a_{n}\}.
\end{equation}
Note that $(B_{k})_{0\leq k\leq m-1}$ are disjoint and $J_{Y}=\bigcup\limits_{k=0}^{m-1}B_{k}$. According to argument above \eqref{319}, we will perform the change of measure on $J_{Y}$ (see Figure 1 in the next page).

\begin{figure}
\centering
\includegraphics[trim={5.5cm 5cm 5cm 8cm}, scale=0.5]{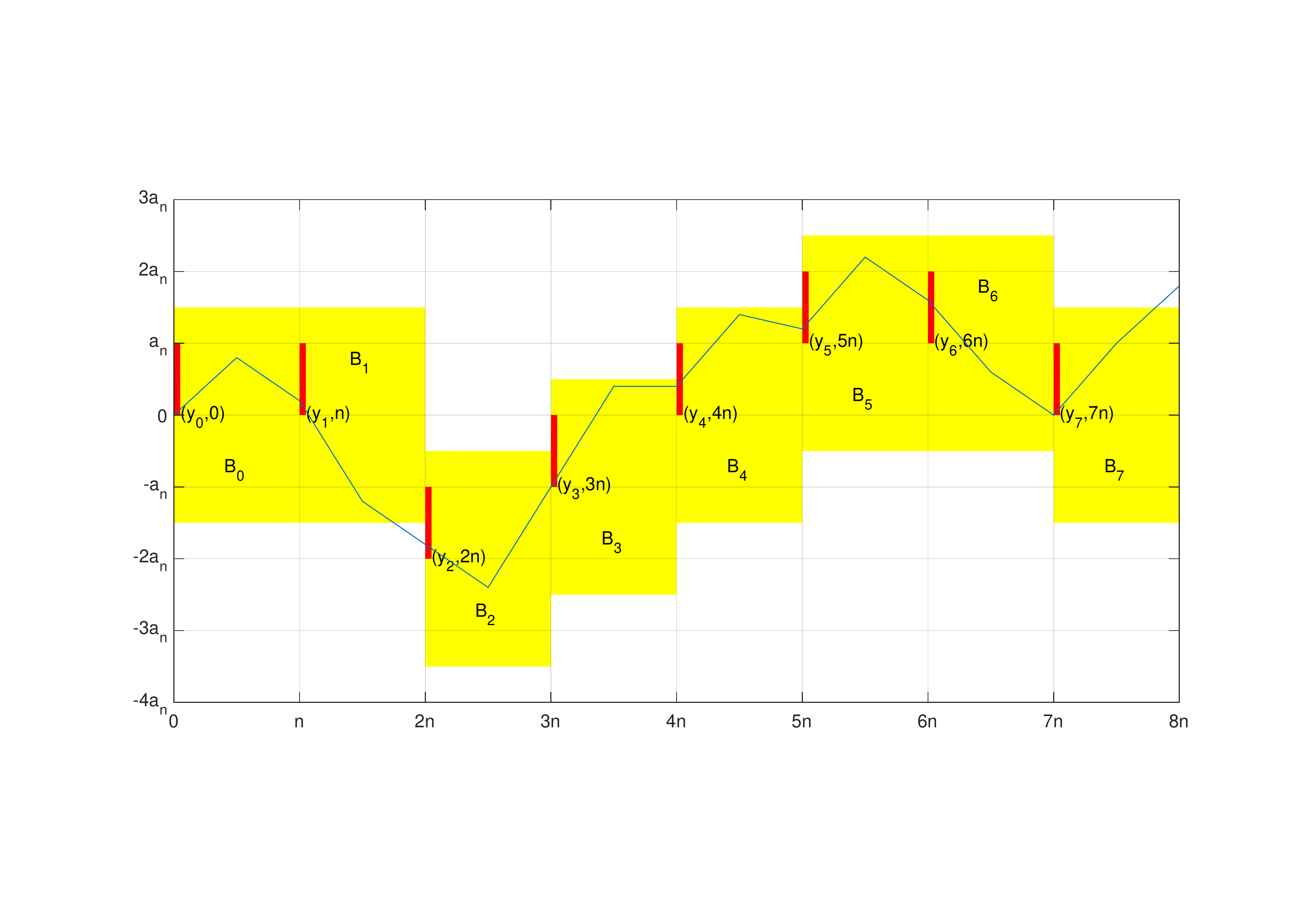}
\caption{This figure represents the coarse-grained version of a trajectory of the random walk $S$. We investigate the random walk $S$ at time $in$, $i=1,\ldots,m$. The bold vertical line segments mean that at time $in$, the random walk $S$ falls in the interval $I_{y_{i}}$, where $y_{i}$ is the vertical coordinate of the lower endpoint of the $i+1$-th bold vertical line segments. The rectangles $B_{k}$ containing $n\times2C_1 a_n$ sites are defined in \eqref{320}, on which we will make change of measure.}
\end{figure}

We define the new measure $\mathbbm{P}_{Y}$, under which $(\omega_{i,x})_{i\geq0,x\in\mathbbm{Z}}$ are independent Gaussian random variables with variance $1$ and expectation $\mathbbm{E}_{Y}[\omega_{1,0}]=-\delta(n)\mathbbm{1}_{(i,x)\in J_{Y}}$, where $\delta(n)$ is a small number and will be determined later. Some direct computation shows that
\begin{equation}\label{321}
\frac{d\mathbbm{P}_{Y}}{d\mathbbm{P}}=\exp\left\{-\sum\limits_{(i,x)\in J_{Y}}\left(\delta(n)\omega_{i,x}+\frac{\delta(n)^{2}}{2}\right)\right\}.
\end{equation}
Then by H$\ddot{\mbox{o}}$lder's inequality,
\begin{equation}\label{322}
\begin{split}
\mathbbm{E}[(\hat{Z}_{(y_{1},\ldots,y_{m})}^{\beta,\omega})^{\theta}]&=\mathbbm{E}_{Y}\left[\frac{d\mathbbm{P}}{d\mathbbm{P}_{Y}}(\hat{Z}_{(y_{1},\ldots,y_{m})}
^{\beta,\omega})^{\theta}\right]\\
&\leq\left(\mathbbm{E}_{Y}\left[\left(\frac{d\mathbbm{P}}{d\mathbbm{P}_{Y}}\right)^{\frac{1}{1-\theta}}\right]\right)^{1-\theta}\left(\mathbbm{E}_{Y}[\hat{Z}
_{(y_{1},\ldots,y_{m})}^{\beta,\omega}]\right)^{\theta}.
\end{split}
\end{equation}
Here
\begin{equation}\label{323}
\left(\mathbbm{E}_{Y}\left[\left(\frac{d\mathbbm{P}}{d\mathbbm{P}_{Y}}\right)^{\frac{1}{1-\theta}}\right]\right)^{1-\theta}=\exp\left
(\frac{|J_{Y}|\theta\delta(n)^{2}}{2(1-\theta)}\right)=\exp\left(\frac{C_{1}a_{n}mn\theta\delta(n)^{2}}{1-\theta}\right).
\end{equation}
To make this term independent of $n$, we can set $\delta(n)=(C_{1}na_{n})^{-\frac{1}{2}}$.

To estimate
\begin{equation}\label{324}
\mathbbm{E}_{Y}[\hat{Z}_{(y_{1},\ldots,y_{m})}^{\beta,\omega}]=\mathbf{E}[\exp(-\beta\delta(n)|\{i:(i,S_{i})\in J_{Y}\}|)\mathbbm{1}_{\{S_{kn}\in I_{y_{k}},
1\leq k\leq m\}}],
\end{equation}
we define
\begin{align}
J&=\{(i,x):i=1,\ldots,n,~|x|\leq(C_{1}-1)a_{n}\},\label{325}\\
\bar{J}&=\{(i,x):i=1,\ldots,n,~|x|\leq(C_{1}-2)a_{n}\},\label{326}
\end{align}
Recall that $J_{Y}=\bigcup\limits_{k=0}^{m-1}B_{k}$ and $B_{k}\cap B_{l}=\emptyset$ for $k\neq l$ by \eqref{320}. We have
\begin{equation}\label{327}
\begin{split}
&\mathbf{E}\left[\exp(-\beta\delta(n)|\{i:(i,S_{i})\in J_{Y}\}|)\mathbbm{1}_{\{S_{kn}\in I_{y_{k}},
1\leq k\leq m\}}\right]\\
=&\mathbf{E}\left[\prod\limits_{k=1}^{m}\exp(-\beta\delta(n)|\{i:(i,S_{i})\in B_{k-1}\}|)\mathbbm{1}_{S_{kn}\in I_{y_{k}}}\right]\\
\leq&\prod\limits_{k=1}^{m}\max\limits_{x\in I_{y_{k-1}}}\mathbf{E}^{x}\left[\exp(-\beta\delta(n)|\{i:(i+(k-1)n,S_{i})\in B_{k-1}\}|)\mathbbm{1}_{S_{n}\in I_{y_{k}}}\right]\\
\leq&\prod\limits_{k=1}^{m}\max\limits_{x\in I_{0}}\mathbf{E}^{x}\left[\exp(-\beta\delta(n)|\{i:(i,S_{i})\in J\}|)\mathbbm{1}_{S_{n}\in I_{y_{k}-y_{k-1}}}\right],
\end{split}
\end{equation}
where the first inequality is due to the Markov property and the last inequality is due to our definition of $I_{k}$ and $J$. Combine \eqref{314}, \eqref{317}, \eqref{318}, and \eqref{323}, it follows that
\begin{equation}\label{328}
\begin{split}
&\log\mathbbm{E}[(\hat{Z}_{N,\beta}^{\omega})^{\theta}]\\
\leq&\log\sum\limits_{y_{1},\ldots,y_{m}\in\mathbbm{Z}}\exp\left(\frac{\theta m}{1-\theta}\right)\left(\mathbbm{E}_{Y}[\hat{Z}_{N,\beta}^{\omega}]\right)^{\theta}\\
\leq&\frac{\theta m}{1-\theta}+\log\sum\limits_{y_{1},\ldots,y_{m}\in\mathbbm{Z}}\left(\prod\limits_{k=1}^{m}\max\limits_{x\in I_{0}}\mathbf{E}^{x}[\exp
(-\beta\delta(n)|\{i:(i,S_{i})\in J\}|)\mathbbm{1}_{S_{n}\in I_{y_{k}-y_{k-1}}}]\right)^{\theta}\\
=&m\left[\frac{\theta}{1-\theta}+\log\sum\limits_{z\in\mathbbm{Z}}(\max\limits_{x\in I_{0}}\mathbf{E}^{x}[\exp(-\beta\delta(n)|\{i:(i,S_{i})\in J\}|)
\mathbbm{1}_{S_{n}\in I_{z}}])^{\theta}\right]
\end{split}
\end{equation}
If we can show that the quantity in the square brackets is smaller than $-1$, then $p(\beta)\leq-\frac{1}{\theta n}$ by \eqref{315}, which will imply very strong disorder. It suffices to show that
\begin{equation}\label{329}
\sum\limits_{z\in\mathbbm{Z}}\max\limits_{x\in I_{0}}\mathbf{E}^{x}[\exp(-\beta\delta(n)|\{i:(i,S_{i})\in J\}|)\mathbbm{1}_{S_{n}\in I_{z}}]^{\theta}
\end{equation}
can be made sufficiently small.

Observe that
\begin{equation}\label{330}
\begin{split}
&\sum\limits_{z\in\mathbbm{Z}}\max\limits_{x\in I_{0}}\mathbf{E}^{x}[\exp(-\beta\delta(n)|\{i:(i,S_{i})\in J\}|)\mathbbm{1}_{S_{n}\in I_{z}}]^{\theta}\\
\leq&\sum\limits_{|y|\geq K}\max_{x\in I_{0}}\mathbf{P}^{x}(S_{n}\in I_{y})^{\theta}+2K\max\limits_{x\in I_{0}}\mathbf{E}^{x}[\exp
(-\beta\delta(n)|\{i:(i,S_{i})\in J\}|)]^{\theta}.
\end{split}
\end{equation}
For the first term,
\begin{equation}\label{331}
\begin{split}
\sum\limits_{|y|\geq K}\max\limits_{x\in I_{0}}\mathbf{P}^{x}(S_{n}\in I_{y})^{\theta}&\leq2\sum\limits_{y=K-2}^{\infty}\mathbf{P}\left(y\leq\frac{S_{n}}{a_{n}}
<y+2\right)^{\theta}\\
&\leq2\sum\limits_{y=K-2}^{\infty}\left(\frac{1}{y^{\gamma}}\mathbf{E}\left|\frac{S_{n}}{a_{n}}\right|^{\gamma}\right)^{\theta}\\
&\leq2C\sum\limits_{y=K-2}^{\infty}y^{-\gamma\theta}.
\end{split}
\end{equation}
The last inequality follows from \cite[Theorem 2.14]{jonsson2007almost} by choosing some $\gamma\in(1,\alpha)$. Therefore, we can fix $\theta$ such that $\gamma\theta>1$ and then choose $K$ large enough such that \eqref{331} is small enough.

For the second term,
\begin{equation}\label{332}
\begin{split}
&2K\max\limits_{x\in I_{0}}\mathbf{E}^{x}\left[\exp(-\beta\delta(n)|\{i:(i,S_{i})\in J\}|)\right]^{\theta}\\
\leq&2K\mathbf{E}\left[\exp(-\beta\delta(n)|\{i:(i,S_{i})\in\bar{J}\}|)\right]^{\theta}\\
\leq&2K[\exp(-n\beta\delta(n))+\mathbf{P}\{\mbox{the random walk goes out of}~\bar{J}\}]
\end{split}
\end{equation}
By choosing a large $C_{1}$, the second term in the square brackets can be made small by the analogue of invariance principle for heavy-tailed random walks. For the first term, notice that
\begin{equation}\label{333}
n\beta\delta(n)=\frac{\beta}{\sqrt{C_{1}}}\sqrt{{\frac{n}{a_{n}}}}=\frac{\beta n^{\frac{\alpha-1}{2\alpha}}}{\sqrt{C_{1}l(n)}}.
\end{equation}
We can choose the smallest $n=n(\beta)$ such that $\beta n^{\frac{\alpha-1}{2\alpha}}l(n)^{-\frac{1}{2}}\geq C_{2}$, where $C_{2}$ is a large constant so that $\exp(-\frac{C_{2}}{\sqrt{C_{1}}})$ is small enough. By our choice of $n$, if follows that
\begin{equation}\label{334}
\lim\limits_{\beta\to0}\beta n^{\frac{\alpha-1}{2\alpha}}l(n)^{-\frac{1}{2}}=C_{2}.
\end{equation}
Therefore,
\begin{equation}\label{335}
n^{\frac{\alpha-1}{2\alpha}}\frac{1}{\sqrt{l(n)}}\sim\frac{C_{2}}{\beta},~~\mbox{as}~\beta\to0.
\end{equation}
Define
\begin{equation}\label{336}
l_{\alpha}(x):=\frac{1}{\sqrt{l(x^{\frac{2\alpha}{\alpha-1}})}}.
\end{equation}
Then $l_{\alpha}(x)$ is also a slowly varying function. We then have
\begin{equation}\label{337}
n^{\frac{\alpha-1}{2\alpha}}l_{\alpha}(n^{\frac{\alpha-1}{2\alpha}})\sim\frac{C_{2}}{\beta},~~\mbox{as}~\beta\to0.
\end{equation}
By \cite[Theorem 1.5.13]{bingham1989regular}, we can find a slowly varying function $l^{\#}_{\alpha}(x)$, such that
\begin{equation}\label{338}
l^{\#}_{\alpha}(xl_{\alpha}(x))\sim\frac{1}{l_{\alpha}(x)},~~\mbox{as}~x\to\infty.
\end{equation}
Therefore,
\begin{equation}\label{339}
\frac{1}{l_{\alpha}(n^{\frac{\alpha-1}{2\alpha}})}\sim l^{\#}_{\alpha}\left(\frac{C_{2}}{\beta}\right),~~\mbox{as}~\beta\to0,
\end{equation}
that is,
\begin{equation}\label{340}
\sqrt{l(n)}\sim l^{\#}_{\alpha}\left(\frac{C_{2}}{\beta}\right),~~\mbox{as}~\beta\to0.
\end{equation}
Combine \eqref{335} and \eqref{340}, and recall that if $l(x)$ is a slowly varying function, then $l(ax)$ and $(l(x))^{\gamma}$ are both slowly varying functions for any $a>0$ and $\gamma\in\mathbbm{R}$ (see\cite{bingham1989regular}). By setting $\varphi=(l^{\#}_{\alpha})^{\frac{\alpha-1}{2\alpha}}$, we then obtain
\begin{equation}\label{341}
\frac{1}{n}\sim C_{2}\beta^{\frac{2\alpha}{\alpha-1}}\varphi\left(\frac{1}{\beta}\right),
\end{equation}
where $\varphi$ is some slowly varying function. Then for some constant $C$
\begin{equation}\label{342}
p(\beta)\leq-\frac{1}{\theta n}\leq-C\beta^{\frac{2\alpha}{\alpha-1}}\varphi\left(\frac{1}{\beta}\right),
\end{equation}
which completes the proof.
\end{proof}
\begin{remark}\label{T306}
In \cite[Proposition 1.5]{lacoin2010new}, Lacoin also gave a lower bound for the free energy of 1-dimensional nearest-neighbor directed polymer with an extra logarithmic term. Later in \cite{alexander2015directed}, the authors proved that one can actually remove the logarithmic term so that the lower bound and the upper bound are consistent (differ up to some prefactor). The proof of the lower bound involves the site percolation. To extend their lower bounds to the long-range model, some property of the long-range percolation may be needed, which has not been systematically studied, however. Besides, the negativity of the upper bound implies very strong disorder, which reflects the qualitative behavior of the polymer chain. Therefore, the upper bound is more significant than the lower bound and we just leave out the lower bound in this paper. Recently, in \cite{caravenna2015universality}, the authors identified the sharp high temperature asymptotic behavior of the shift of the critical point for the pinning model with exponent $\alpha\in(\frac{1}{2},1)$. We expect that their approach is also applicable to the long-range directed polymer with $\alpha\in(1,2]$.
\end{remark}
\begin{remark}\label{T307}
We continue the discussion in Remark \ref{T114}. Although we have given some equivalent conditions for recurrence of heavy-tailed random walks in Proposition \ref{T301}, for stable exponent $\alpha=1$, we have not deduced very strong disorder for all $\beta>0$ from the recurrence of the random walk $S$. The reason is that the slowly varying function $L(x)$ is subtle and the tail distribution of the random walk $S$ has much slower decay than that of a simple random walk. Therefore, some more delicate techniques are needed. In Berger and Lacoin's recent papers \cite{berger2015high, berger2015pinning}, they developed a more elaborate change of measure procedure. By that method, the authors identified the sharp high temperature asymptotic behavior for the nearest-neighbor directed polymer in $\mathbbm{Z}^{2+1}$ in \cite{berger2015high}, and the sharp asymptotics on the critical point shift for the pinning of one dimensional simple random walk. Note that $d=2$ is the critical dimension for the existence of the weak disorder regime in the nearest-neighbor directed polymer model on $\mathbbm{Z}^{d+1}$, and the case $\alpha=\frac{1}{2}$ for pinning model is critical for whether the disorder is relevant. Hence, we believe that their new method can also be applied to provide the asymptotic behavior of the free energy for long-range directed polymer model in the critical case $\alpha=1$. This paper does not include that case because it is quite involved and hence should be treated
separately.
\end{remark}
\section{Proof of Theorem \ref{T117}}
We will first extend the key lemma \cite[Lemma 5.3]{vargas2007strong} so that it holds not only for finite $(\eta_{i})_{i=1}^{n}$, but also for countable many $(\eta_{i})_{i\geq1}$.
\begin{lemma}\label{T401}
Denote $\Lambda=\{(\lambda_{i})_{i\geq1}\subset[0,1]^{\mathbbm{N}}:\sum\limits_{i=1}^{\infty}\lambda_{i}=1\}$, and let $(\eta_{i})_{i\geq1}$ be an i.i.d.sequence of positive random variables such that $\mathbbm{E}[|\log\eta_{1}|]<\infty$. Then for any positive integer $k$, we have
\begin{equation}\label{401}
\inf\limits_{(\lambda_{i})\in\Lambda\atop\sup(\lambda_{i})\leq\frac{1}{k}}\mathbbm{E}\left[\log\left(\sum\limits_{i=1}^{\infty}\lambda_{i}
\eta_{i}\right)\right]=\mathbbm{E}\left[\log\left(\frac{1}{k}\sum\limits_{i=1}^{k}\eta_{i}\right)\right].
\end{equation}
\end{lemma}
\begin{proof}
We prove the lemma by contradiction. Assume that
\begin{equation}\label{402}
\inf\limits_{(\lambda_{i})\in\Lambda\atop\sup(\lambda_{i})\leq\frac{1}{k}}\mathbbm{E}\left[\log\left(\sum\limits_{i=1}^{\infty}
\lambda_{i}\eta_{i}\right)\right]<\mathbbm{E}\left[\log\left(\frac{1}{k}\sum\limits_{i=1}^{k}\eta_{i}\right)\right],
\end{equation}
then we can find a sequence $(\bar{\lambda}_{i})$ such that
\begin{equation}\label{403}
\mathbbm{E}\left[\log\left(\sum\limits_{i=1}^{\infty}\bar{\lambda}_{i}\eta_{i}\right)\right]
<\mathbbm{E}\left[\log\left(\frac{1}{k}\sum\limits_{i=1}^{k}\eta_{i}\right)\right].
\end{equation}
Note that there are only finite many $\bar{\lambda}_{i}$'s that equal to $\frac{1}{k}$, and by continuity, we can adjust those $\bar{\lambda}_{i}$'s if necessary such that $\sup_{i}\bar{\lambda}_{i}=\epsilon<\frac{1}{k}$ and \eqref{403} still holds. For any fixed integer $n$ which is large enough such that $\Lambda_{n}=\sum\limits_{i}^{n}\bar{\lambda}_{i}>\epsilon k$, we set $\tilde{\lambda}_{i}=\frac{\bar{\lambda}_{i}}{\Lambda_{n}}$ for $1\leq i\leq n$. Then,
\begin{equation}\label{404}
\mathbbm{E}\left[\log\left(\sum\limits_{i=1}^{\infty}\bar{\lambda}_{i}\eta_{i}\right)\right]\geq
\mathbbm{E}\left[\log\left(\sum\limits_{i=1}^{n}\tilde{\lambda}_{i}\eta_{i}\right)\right]+\log\Lambda_{n}
\geq\mathbbm{E}\left[\log\left(\frac{1}{k}\sum\limits_{i=1}^{k}\eta_{i}\right)\right]+\log\Lambda_{n},
\end{equation}
where the first inequality is due to the positivity of $\eta_{i}$ and the second inequality holds by \cite[Lemma 5.3]{vargas2007strong} since $\sup\limits_{1\leq i\leq n}\tilde{\lambda}_{i}\leq\frac{1}{k}$ and $\sum\limits_{i=1}^{n}\tilde{\lambda_{i}}=1$. Let $n$ tend to infinity, then $\log\Lambda_{n}$ tends to 0 and \eqref{404} contradicts \eqref{403}.
\end{proof}
\begin{proof}[Proof of Theorem \ref{T117}]
We follow the same strategy of proof as that for \cite[Theorem 3.7]{vargas2007strong} in the nearest-neighbor case. We will decompose $N^{-1}\log\hat{Z}_{N,\beta}^{\omega}$ to construct a martingale by successively conditioning on $\mathcal{G}_{N}$, which is the $\sigma$-field generated by $(\omega_{i,x})_{1\leq i\leq N,x\in\mathbbm{Z}}$. First, define
\begin{equation}\label{405}
A_{N,\beta}^{\epsilon}=\{\omega:\sup\limits_{x\in\mathbbm{Z}}\mathbf{P}_{N-1,\beta}^{\omega}(S_{N}=x)>\epsilon\}.
\end{equation}
Then
\begin{equation}\label{406}
\begin{split}
\frac{\log Z_{N,\beta}^{\omega}}{N}=&\frac{1}{N}\sum\limits_{j=1}^{N}\log\frac{Z_{j,\beta}^{\omega}}{Z_{j-1,\beta}^{\omega}}\\
=&\frac{1}{N}\sum\limits_{j=1}^{N}\mathbbm{1}_{A_{j,\beta}^{\epsilon}}\log\left(\sum\limits_{x\in\mathbbm{Z}}\mathbf{P}_{j-1,\beta}^{\omega}(S_{j}=x)\exp
(\beta\omega_{j,x})\right)\\
+&\frac{1}{N}\sum\limits_{j=1}^{N}\mathbbm{1}_{(A_{j,\beta}^{\epsilon})^{c}}\log\left(\sum\limits_{x\in\mathbbm{Z}}\mathbf{P}_{j-1,\beta}^{\omega}(S_{j}=x)\exp
(\beta\omega_{j,x})\right)
\end{split}
\end{equation}
Note that in the second term of the right-hand side of \eqref{406}, $\sup\limits_{x\in\mathbbm{Z}}\mathbf{P}_{N-1,\beta}^{\omega}(S_{N}=x)\leq\epsilon$. Hence, we can apply Lemma \ref{T401} to this term later.

Define $\mathcal{G}_{N}$-martingales
\begin{equation}\label{407}
\begin{split}
M_{N}:=&\sum\limits_{j=1}^{N}\mathbbm{1}_{(A_{j,\beta}^{\epsilon})^{c}}\log\left(\sum\limits_{x\in\mathbbm{Z}}\mathbf{P}_{j-1,\beta}^{\omega}(S_{j}=x)
\exp(\beta\omega_{j,x})\right)\\
&-\sum\limits_{j=1}^{N}\mathbbm{1}_{(A_{j,\beta}^{\epsilon})^{c}}\mathbbm{E}\left[\left.\log\left(\sum\limits_{x\in\mathbbm{Z}}\mathbf{P}_{j-1,\beta}^{\omega}
(S_{j}=x)\exp(\beta\omega_{j,x})\right)\right|\mathcal{G}_{j-1}\right],
\end{split}
\end{equation}
and
\begin{equation}\label{408}
\begin{split}
L_{N}:=&\sum\limits_{j=1}^{N}\mathbbm{1}_{A_{j,\beta}^{\epsilon}}\log\left(\sum\limits_{x\in\mathbbm{Z}}\mathbf{P}_{j-1,\beta}^{\epsilon}(S_{j}=x)
\exp(\beta\omega_{j,x})\right)\\
&-\sum\limits_{j=1}^{N}\mathbbm{1}_{A_{j,\beta}^{\epsilon}}\mathbbm{E}\left[\left.\log\left(\sum\limits_{x\in\mathbbm{Z}}\mathbf{P}_{j-1,\beta}^{\omega}(S_{j}=x)
\exp(\beta\omega_{j,x})\right)\right|\mathcal{G}_{j-1}\right]
\end{split}
\end{equation}
Then
\begin{equation}\label{409}
\begin{split}
&\frac{1}{N}\sum\limits_{j=1}^{N}\mathbbm{1}_{A_{j,\beta}^{\epsilon}}\log\left(\sum\limits_{x\in\mathbbm{Z}}\mathbf{P}_{j-1,\beta}^{\omega}(S_{j}=x)\exp
(\beta\omega_{j,x})\right)\\
=&\frac{L_{N}}{N}+\frac{1}{N}\sum\limits_{j=1}^{N}\mathbbm{1}_{A_{j,\beta}^{\epsilon}}\mathbbm{E}\left[\left.\log\left(\sum\limits_{x\in\mathbbm{Z}}
\mathbf{P}_{j-1,\beta}^{\epsilon}(S_{j}=x)\exp(\beta\omega_{j,x})\right)\right|\mathcal{G}_{j-1}\right]\\
\geq&\frac{L_{N}}{N}+\beta\mathbbm{E}[\omega_{1,0}]\cdot\frac{1}{N}\sum\limits_{j=1}^{N}\mathbbm{1}_{A_{j}^{\epsilon,\beta}}=\frac{L_{N}}{N}.
\end{split}
\end{equation}
by Jensen's inequality. And
\begin{equation}\label{410}
\begin{split}
&\frac{1}{N}\sum\limits_{j=1}^{N}\mathbbm{1}_{(A_{j,\beta}^{\epsilon})^{c}}\log\left(\sum\limits_{x\in\mathbbm{Z}}\mathbf{P}_{j-1,\beta}^{\omega}(S_{j}=x)\exp
(\beta\omega_{j,x})\right)\\
=&\frac{M_{N}}{N}+\frac{1}{N}\sum\limits_{j=1}^{N}\mathbbm{1}_{(A_{j,\beta}^{\epsilon})^{c}}\mathbbm{E}\left[\left.\log\left(\sum\limits_{x\in\mathbbm{Z}}
\mathbf{P}_{j-1,\beta}^{\omega}(S_{j}=x)\exp(\beta\omega_{j,x})\right)\right|\mathcal{G}_{j-1}\right]\\
\geq&\frac{M_{N}}{N}+\frac{1}{N}\sum\limits_{j=1}^{N}\mathbbm{1}_{(A_{j,\beta}^{\epsilon})^{c}}\mathbbm{E}\left[\log\left(\epsilon\sum\limits_{i=1}
^{\frac{1}{\epsilon}}\exp(\beta\omega_{i,0})\right)\right]
\end{split}
\end{equation}
by \eqref{401} with $(\exp(\beta\omega_{j,x}))_{j\geq0, x\in\mathbbm{Z}}$ and $(\mathbf{P}_{j-1,\beta}^{\omega}(S_{j}=x))_{x\in\mathbbm{Z}}$ playing respectively the role of $(\eta_{i})_{i\leq1}$ and $(\lambda_{i})_{i\geq1}$ in \eqref{401}. We obtain
\begin{equation}\label{411}
\frac{\log Z_{N,\beta}^{\omega}}{N}-\frac{M_{N}}{N}-\frac{L_{N}}{N}\geq\left(\frac{1}{N}\sum\limits_{j=1}^{N}\mathbbm{1}_{(A_{j,\beta}^{\epsilon})^{c}}\right)
\mathbbm{E}\left[\log\left(\epsilon\sum\limits_{i=1}^{\frac{1}{\epsilon}}\exp(\beta\omega_{i,0})\right)\right].
\end{equation}
We will then prove that $\frac{M_{N}}{N}$ and $\frac{L_{N}}{N}$ tend to $0$ as $N$ tends to infinity by applying the following theorem \cite[Theorem 2.19]{hall2014martingale}
\begin{theorem}[Hall-Heyde~\cite{hall2014martingale}]\label{T402}
Let $(Y_{n})_{n\geq1}$ be a sequence of random variables and $(\mathcal{F}_{n})_{n\geq1}$ an increasing sequence of $\sigma$-fields with $Y_{n}$ measurable with respect to $\mathcal{F}_{n}$ for each $n$. Let $Y$ be a random variable and $c$ a constant such that $\mathbbm{E}|Y|<\infty$ and $\mathbbm{P}(|Y_{n}|>x)\leq c\mathbbm{P}(|Y|>x)$ for each $x>0$ and $n\geq1$. Then
\begin{equation}\label{412}
n^{-1}\sum\limits_{i=1}^{n}[Y_{i}-\mathbbm{E}[Y_{i}|\mathcal{F}_{i-1}]]\overset{\mathbbm{P}}{\to}0~~\mbox{as}~n\to\infty.
\end{equation}
If $\mathbbm{E}[|Y|\log^{+}|Y|]<\infty$, then the convergence in probability in \eqref{412} can be strengthen to almost sure convergence.
\end{theorem}
First, by Jensen's inequality, we have
\begin{equation}\label{413}
\beta\sum\limits_{x\in\mathbbm{Z}}\mathbf{P}_{j-1,\beta}^{\omega}(S_{j}=x)\omega_{j,x}\leq\log\left(\sum\limits_{x\in\mathbbm{Z}}\mathbf{P}_{j-1,\beta}^{\omega}
(S_{j}=x)\exp(\beta\omega_{j,x})\right).
\end{equation}
And by using that $\log x\leq x^{\frac{1}{\theta}}$ for $1<\theta\leq e$, we have
\begin{equation}\label{414}
\log\left(\sum\limits_{x\in\mathbbm{Z}}\mathbf{P}_{j-1,\beta}^{\omega}(S_{j}=x)\exp(\beta\omega_{j,x})\right)\leq\left(\sum\limits_{x\in\mathbbm{Z}}
\mathbf{P}_{j-1,\beta}^{\omega}(S_{j}=x)\exp(\beta\omega_{j,x})\right)^{\frac{1}{\theta}}.
\end{equation}
Then, by applying \eqref{413} when $\log\left(\sum\limits_{x\in\mathbbm{Z}}\mathbf{P}_{j-1,\beta}^{\omega}
(S_{j}=x)\exp(\beta\omega_{j,x})\right)<0$ and \eqref{414} when\\
$\log\left(\sum\limits_{x\in\mathbbm{Z}}\mathbf{P}_{j-1,\beta}^{\omega}
(S_{j}=x)\exp(\beta\omega_{j,x})\right)>0$, it follows that for all $j$,
\begin{equation}\label{415}
\begin{split}
&\mathbbm{E}\left|\log\left(\sum\limits_{x\in\mathbbm{Z}}\mathbf{P}_{j-1,\beta}^{\omega}(S_{j}=x)\exp(\beta\omega_{j,x})\right)\right|^{\theta}\\
\leq&\beta^{\theta}\mathbbm{E}\left(\sum\limits_{x\in\mathbbm{Z}}\mathbf{P}_{j-1,\beta}^{\omega}(S_{j}=x)|\omega_{j,x}|\right)^{\theta}
+\mathbbm{E}\left[\sum\limits_{x\in\mathbbm{Z}}\mathbf{P}_{j-1,\beta}^{\omega}(S_{j}=x)\exp(\beta\omega_{j,x})\right]\\
\leq&\beta^{\theta}\mathbbm{E}|\omega_{1,0}|^{\theta}+\exp(\lambda(\beta))=C.
\end{split}
\end{equation}
Let $\mathbbm{1}_{(A_{j,\beta}^{\epsilon})^{c}}\log\left(\sum\limits_{x\in\mathbbm{Z}}\mathbf{P}_{j-1,\beta}^{\omega}(S_{j}=x)
\exp(\beta\omega_{j,x})\right)$ and $\mathbbm{1}_{A_{j,\beta}^{\epsilon}}\log\left(\sum\limits_{x\in\mathbbm{Z}}\mathbf{P}_{j-1,\beta}^{\omega}(S_{j}=x)
\exp(\beta\omega_{j,x})\right)$ play the role $Y_{j}$ in Theorem \ref{T402}, and define a random variable $Y$ such that for all $x>C^{\frac{1}{\theta}}$,
\begin{equation}\label{416}
\mathbbm{P}(|Y|>x)=\frac{C}{x^{\theta}},
\end{equation}
where $C$ is the same as that in \eqref{415}. Then,
\begin{equation}\label{417}
\lim\limits_{N\to\infty}\frac{M_{N}}{N}=\lim\limits_{N\to\infty}\frac{L_{N}}{N}=0,~~\mbox{in}~\mathbbm{P}\mbox{-probability}.
\end{equation}
Note that, recalling $\theta>1$ and by the definition \eqref{416}, $\mathbbm{E}[|Y|\log^{+}|Y|]<\infty$. Therefore, the convergence in \eqref{417} can be strengthened to almost sure convergence. By taking limits on both sides of \eqref{411}, we have
\begin{equation}\label{418}
\varlimsup\limits_{N\to\infty}\frac{1}{N}\sum\limits_{j=1}^{N}\mathbbm{1}_{(A_{j,\beta}^{\epsilon})^{c}}\leq\frac{F(\beta)}
{\mathbbm{E}\left[\log\left(\epsilon\sum\limits_{i=1}^{\frac{1}{\epsilon}}\exp(\beta\omega_{i,0})\right)\right]},~\mathbbm{P}\mbox{-a.s.},
\end{equation}
where $F(\beta)$ is the free energy of the system by \eqref{110}. Let $\epsilon$ tend to $0$ along the sequence $(\frac{1}{k})_{k\geq1}$. By Jensen's inequality, the law of large numbers and Fatou's lemma, it is not hard to see
\begin{equation}\label{419}
\lim\limits_{\epsilon\to0}\mathbbm{E}\left[\log\left(\epsilon\sum\limits_{i=1}^{\frac{1}{\epsilon}}\exp(\beta\omega_{i,0})\right)\right]
=\lambda(\beta)>F(\beta).
\end{equation}
The last inequality is due to our very strong disorder assumption.

Hence, we can choose $\epsilon$ small enough such that
\begin{equation}\label{420}
\mathbbm{E}\left[\log\left(\epsilon\sum\limits_{i=1}^{\frac{1}{\epsilon}}\exp(\beta\omega_{i,0})\right)\right]>F(\beta).
\end{equation}
Then by \eqref{414}, for $\mathbbm{P}$-a.s.,
\begin{equation}\label{421}
\varlimsup\limits_{N\to\infty}\frac{1}{N}\sum\limits_{j=1}^{N}\mathbbm{1}_{(A_{j,\beta}^{\epsilon})^{c}}<1\Leftrightarrow
\varliminf\limits_{N\to\infty}\frac{1}{N}\sum\limits_{j=1}^{N}\mathbbm{1}_{(A_{j,\beta}^{\epsilon})}>0
\end{equation}
Recall the definition of $\mathcal{A}_{N,\beta}^{\epsilon,\omega}$ and $A_{N,\beta}^{\epsilon}$ in \eqref{124} and \eqref{405}, and then \eqref{421} implies \eqref{125}.
\end{proof}
\section{Proof of Theorem \ref{T118}}
The basic idea of the proof is to compare the entropy cost and the energy gain when a heavy-tailed random walk introduced by \eqref{126} and \eqref{127} stays in a distance of $\mathcal{O}\left(\frac{N}{(\log N)^{2}}\right)$ away from the origin. It can be seen that
\begin{equation}\label{501}
Z_{N,\beta}^{\omega}=\sum\limits_{S}\exp(-\beta H_{N}^{\omega}(S))\mathbf{P}(S),
\end{equation}
where $H_{N}^{\omega}(S)$ is the energy introduced by \eqref{107}. For technical feasibility, we may study the second half of the trajectory of the random walk, i.e., $(S_{N/2},\ldots,S_{N})$. On one hand, if the second half of $S_{N}$ stays in a distance $N/(\log N)^{2}$, which is $\gg N^{\frac{1}{\alpha}}$ for $\alpha\in(1,2]$ and makes $\mathbf{P}(S)$ very small, then there is a significant entropy cost. On the other hand, with some random variable $Y$, we may write $\exp(-\beta H_{N}^{\omega}(S))\approx\exp(-\sqrt{N}Y)$, which fluctuates dramatically. Therefore, it is possible that we can find some block with very high energy on $\mathbbm{Z}$ in a distance of $\mathcal{O}\left(\frac{N}{(\log N)^{2}}\right)$ away from the origin, and if the energy gain wins the entropy cost, then the random walk is likely to stay in that block instead of somewhere near the origin.

Our proof consists of two parts. We will first investigate the energy gain. However, we will not estimate the energy directly. Instead, we will compare the contribution to the partition function from the environment on different blocks. In order to do that, we will use a change of measure argument developed in \cite{lacoin2011influence}, since it is more likely to extend to the model with some general environment and it is much shorter than the method used in \cite{bezerra2008superdiffusivity}. Then we need to compute the entropy cost, which will be done by an estimate on a Radon-Nikodym derivative, although it is not as accurate as the Girsanov Theorem used in \cite{bezerra2008superdiffusivity, lacoin2011influence}.
\begin{proof}[Proof of Theorem \ref{T118}]
Without loss of generality, we can assume that the integer $N$ is always even throughout the proof, such that we can omit many $"\lfloor\cdot\rfloor"$ symbols to make the proof more readable.

For any given $\epsilon>0$, to be consistent with \eqref{128}, we denote
\begin{equation}\label{502}
J_{N}=\left(-\frac{\beta^{2}N}{4(\alpha+1+\epsilon)^{2}(\log N)^{2}},\frac{\beta^{2}N}{4(\alpha+1+\epsilon)^{2}(\log N)^{2}}\right)\cap\mathbbm{Z}.
\end{equation}
Then we can define a change of measure from $\mathbbm{P}$ to a new probability measure $\mathbbm{\hat{P}}$ with Ladon-Nikodym derivative
\begin{equation}\label{503}
\frac{\mbox{d}\mathbbm{\hat{P}}}{\mbox{d}\mathbbm{P}}:=\exp\left(-W-\frac{1}{2}\right),
\end{equation}
where
\begin{equation}\label{504}
W=\frac{\sum\limits_{n=\frac{N}{2}+1}^{N}\sum\limits_{x\in J_{N}}\omega_{n,x}}{\sqrt{\frac{N}{2}|J_{N}|}}.
\end{equation}
It is not hard to check that $\hat{\omega}:=(\hat{\omega}_{i,x})_{(i,x)\in\mathbbm{N}\times\mathbbm{Z}}$ defined by
\begin{equation}\label{505}
\hat{\omega}_{i,x}=\omega_{i,x}+\mathbbm{1}_{\left\{(i,x)\in\left[\frac{N}{2}+1,N\right]\times J_{N}\right\}}\left(\frac{N}{2}|J_{N}|\right)^{-\frac{1}{2}}
\end{equation}
is a family of i.i.d. standard Gaussian random variables under $\mathbbm{\hat{P}}$. Probability measure $\mathbbm{\hat{P}}$ has some important property. Firstly, it makes the random environment on $\left[\frac{N}{2},N\right]\times J_{N}$ become less attractive to the random walk. Secondly, it does not differ from $\mathbbm{P}$ too much, which can be seen by the following application of the H$\ddot{\rm{o}}$lder inequality:
\begin{equation}\label{506}
\mathbbm{P}(A)=\mathbbm{\hat{E}}\left[\frac{\rm{d}\mathbbm{P}}{\rm{d}\mathbbm{\hat{P}}}\mathbbm{1}_{A}\right]\leq\sqrt{\mathbbm{E}
\left[\frac{\rm{d}\mathbbm{P}}{\rm{d}\mathbbm{\hat{P}}}\right]}\sqrt{\mathbbm{\hat{P}}(A)}\leq\sqrt{e\mathbbm{\hat{P}}(A)}.
\end{equation}
Then for any $p\in(0,1]$, we have
\begin{equation}\label{507}
\begin{split}
&\mathbbm{P}\left(\mathbf{P}_{N,\beta}^{\omega}\left(\max\limits_{1\leq n\leq N}|S_{n}|<\frac{1}{2}|J_{N}|\right)\geq p\right)\leq
\sqrt{e\mathbbm{\hat{P}}\left(\mathbf{P}_{N,\beta}^{\omega}\left(\max\limits_{1\leq n\leq N}|S_{n}|<\frac{1}{2}|J_{N}|\right)\geq p\right)}\\
=&\sqrt{e\mathbbm{\hat{P}}\left(\frac{\mathbf{E}\left[\exp\left(\beta\sum\limits_{n=1}^{N}\omega_{n,S_{n}}\right)\mathbbm{1}_{\{|S_{n}|
<\frac{1}{2}|J_{N}|,~\forall n\in[1,N]\}}\right]}{\mathbf{E}\left[\exp\left(\beta\sum\limits_{n=1}^{N}\omega_{n,S_{n}}\right)\right]}\geq p\right)}.
\end{split}
\end{equation}
In order to deal with the last term in \eqref{507}, we partition all integer $\mathbbm{Z}$ by
\begin{equation}\label{508}
I_{N}^{k}=\left[(2k-1)L, (2k+1)L\right)\cap\mathbbm{Z},~\forall k\in\mathbbm{Z},
\end{equation}
where
\begin{equation}\label{509}
L=\left\lfloor\frac{\beta^{2}N}{4(\alpha+1+\epsilon_{0})^{2}(\log N)^{2}}\right\rfloor
\end{equation}
with some $\epsilon_{0}\in(0,\epsilon)$ so that when $N$ is large enough, $J_{N}\subset I_{N}^{0}$. Note that under this partition, only those $\omega_{i,x}$'s with $(i,x)\in[\frac{N}{2}+1,N]\times I_{N}^{0}$ is influenced by changing measure from $\mathbbm{P}$ to $\mathbbm{\hat{P}}$.
We define
\begin{equation}\label{510}
Z_{N,\beta}^{\omega}(k):=\mathbf{E}\left[\exp\left(\beta\sum\limits_{n=1}^{N}\omega_{n,S_{n}}\right)\mathbbm{1}_{\left\{S_{n}\in I_{N}^{k},~\forall n\in\left[\frac{N}{2}+1,N\right]\right\}}\right]
\end{equation}
and
\begin{equation}\label{511}
\hat{Z}_{N,\beta}^{\omega}:=\mathbf{E}\left[\exp\left(\beta\sum\limits_{n=1}^{N}\omega_{n,S_{n}}\right)\mathbbm{1}_{\{|S_{n}|
<\frac{1}{2}|J_{N}|,~\forall n\in[1,N]\}}\right].
\end{equation}
Since $I_{N}^{k}$ and $I_{N}^{j}$ are disjoint for $k\neq j$, we have for any positive integer $M$,
\begin{equation}\label{512}
Z_{N,\beta}^{\omega}\geq\sum\limits_{k\in\{-M,\ldots,M\}\setminus\{0\}}Z_{N,\beta}^{\omega}(k).
\end{equation}
Then we can bound the last term in \eqref{507} by
\begin{equation}\label{513}
\begin{split}
&\mathbbm{\hat{P}}\left(\frac{\mathbf{E}\left[\exp\left(\beta\sum\limits_{n=1}^{N}\omega_{n,S_{n}}\right)\mathbbm{1}_{\{|S_{n}|
<\frac{1}{2}|J_{N}|,~\forall n\in[1,N]\}}\right]}{\mathbf{E}\left[\exp\left(\beta\sum\limits_{n=1}^{N}\omega_{n,S_{n}}\right)\right]}\geq p\right)\\
\leq&\mathbbm{\hat{P}}\left(\frac{\hat{Z}_{N,\beta}^{\omega}}{\sum\limits_{k\in\{-M,\ldots,M\}\setminus\{0\}}Z_{N,\beta}^{\omega}(k)}\geq p\right)\\
=&\mathbbm{\hat{P}}\left(\exp\left(-\beta\frac{N}{2}\left(\frac{N}{2}|J_{N}|\right)^{-\frac{1}{2}}\right)\frac{\hat{Z}_{N,\beta}^{\hat{\omega}}}
{\sum\limits_{k\in\{-M,\ldots,M\}\setminus\{0\}}Z_{N,\beta}^{\omega}(k)}\geq p\right)\\
=&\mathbbm{P}\left(\exp\left(-\beta\frac{N}{2}\left(\frac{N}{2}|J_{N}|\right)^{-\frac{1}{2}}\right)\frac{\hat{Z}_{N,\beta}^{\omega}}
{\sum\limits_{k\in\{-M,\ldots,M\}\setminus\{0\}}Z_{N,\beta}^{\omega}(k)}\geq p\right),
\end{split}
\end{equation}
where in the first equality, we change $\omega$ to $\hat{\omega}$ and the last equality results from the property $\mathcal{L}_{\mathbbm{P}}(\omega)=\mathcal{L}_{\mathbbm{\hat{P}}}(\hat{\omega})$. The proof will be completed by the following proposition, whose proof will be given later.
\begin{proposition}\label{T501}
For any $\epsilon>0$, there exists some constant $C>0$, such that for any positive integer $M$ and large enough even integer $N$, we have
\begin{equation}\label{514}
\sum\limits_{k\in\{-M,\ldots,M\}\setminus\{0\}}Z_{N,\beta}^{\omega}(k)\geq C(MN)^{-(\alpha+1+\frac{\epsilon}{2})}Z_{N,\beta}^{\omega}(0)
\end{equation}
with $\mathbbm{P}$-probability greater than $1-\frac{1}{2M}$.
\end{proposition}
By $J_{N}\subset I_{N}^{0}$ and Proposition \ref{T501},
\begin{equation}\label{515}
\frac{\hat{Z}_{N,\beta}^{\omega}}
{\sum\limits_{k\in\{-M,\ldots,M\}\setminus\{0\}}Z_{N,\beta}^{\omega}(k)}\leq\frac{Z_{N,\beta}^{\omega}(0)}
{\sum\limits_{k\in\{-M,\ldots,M\}\setminus\{0\}}Z_{N,\beta}^{\omega}(k)}\leq C(MN)^{\alpha+1+\frac{\epsilon}{2}}
\end{equation}
with $\mathbbm{P}$-probability greater than $1-\frac{1}{2M}$. Note that by our choice of $J_{N}$,
\begin{equation}\label{516}
\exp\left(-\beta\frac{N}{2}\left(\frac{N}{2}|J_{N}|\right)^{-\frac{1}{2}}\right)\sim N^{-(\alpha+1+\epsilon)}.
\end{equation}
Combine \eqref{507}, \eqref{513}, \eqref{516}, \eqref{517} and choosing $p=N^{-\frac{\epsilon}{4}}$ in \eqref{507}, and then we have
\begin{equation}\label{517}
\mathbbm{P}\left(\mathbf{P}_{N,\beta}^{\omega}\left(\max\limits_{1\leq n\leq N}|S_{n}|<\frac{1}{2}|J_{N}|\right)\geq N^{-\frac{\epsilon}{4}}\right)\leq\sqrt{\frac{e}{2M}}
\end{equation}
when $N$ is large enough. Thus
\begin{equation}\label{518}
\mathbbm{E}\left[\mathbf{P}_{N,\beta}^{\omega}\left(\max\limits_{1\leq n\leq N}|S_{n}|<\frac{1}{2}|J_{N}|\right)\right]\leq\sqrt{\frac{e}{2M}}+N^{-\frac{\epsilon}{4}}.
\end{equation}
By sending $N$ to infinity and then sending $M$ to infinity, we finish the proof of Theorem \ref{T118}.
\end{proof}
Now we prove Proposition \ref{T501}, which gives an estimate on the entropy cost for a random walk staying in the blocks which are far away from the origin.
\begin{proof}[Proof of Proposition \ref{T501}]
For all $k\in\{-M,\ldots,M\}$, by recalling $L$ in \eqref{509} and $I_{N}^{k}$ in \eqref{508}, we define
\begin{equation}\label{519}
h_{N}(n,k)=\begin{cases}
0,~&\mbox{for}~1\leq n\leq\frac{N}{2},\\
2kL,~&\mbox{for}~\frac{N}{2}+1\leq n\leq N,
\end{cases}
\end{equation}
and
\begin{equation}\label{520}
\overline{Z}_{N,\beta}^{\omega}(k):=\mathbf{E}\left[\exp\left(\beta\sum\limits_{n=1}^{N}\omega_{n,S_{n}+h_{N}(n,k)}\right)\mathbbm{1}_{\left
\{S_{n}\in I_{N}^{0},~\forall n\in\left[\frac{N}{2}+1,N\right]\right\}}\right].
\end{equation}
When $S_{n}\in I_{N}^{0}$ for all $n\in\left[\frac{N}{2}+1,N\right]$, $\left\{(\omega_{n,S_{n}+h_{N}(n,k)})_{n\in\left[\frac{N}{2}+1,N\right]}\right\}_{k\in\{-M,\ldots,M\}}$ are independent families for different $k$. Hence, it is easy to show that $\mathbbm{P}(\overline{Z}_{N,\beta}^{\omega}(k)=\overline{Z}_{N,\beta}^{\omega}(j))=0$ for $k\neq j$ and $(\overline{Z}_{N,\beta}^{\omega}(k))_{k\in\{-M,\ldots,M\}}$ is an exchangeable sequence. Therefore,
\begin{equation}\label{521}
\mathbbm{P}\left(\overline{Z}_{N,\beta}^{\omega}(0)=\max\limits_{k\in\{-M,\ldots,M\}}\overline{Z}_{N,\beta}^{\omega}(k)\right)=\frac{1}{2M+1}
\end{equation}
Note that $\overline{Z}_{N,\beta}^{\omega}(0)=Z_{N,\beta}^{\omega}(0)$ and we need to compare $\overline{Z}_{N,\beta}^{\omega}(k)$ with $Z_{N,\beta}^{\omega}(k)$ for $k\neq0$. By writing $\overline{S}_{n}=S_{n}-h_{N}(n,k)$, we have
\begin{equation}\label{522}
Z_{N,\beta}^{\omega}=\mathbf{E}\left[\exp\left(\beta\sum\limits_{n=1}^{N}\omega_{n,\overline{S}_{n}+h_{N}(n,k)}\right)\mathbbm{1}_{\left
\{\overline{S}_{n}\in I_{N}^{0},~\forall n\in\left[\frac{N}{2}+1,N\right]\right\}}\right].
\end{equation}
We can complete the proof with the help of the following lemma.
\begin{lemma}\label{T502}
Define a sequence of random variables $(\overline{X}_{n})_{1\leq n\leq N}$ by
\begin{equation}\label{523}
\overline{X}:=\begin{cases}
X_{n},&~\rm{for}~n\neq\frac{N}{2}+1,\\
X_{n}-h_{N}(n,k),&~\rm{for}~n=\frac{N}{2}+1.
\end{cases}
\end{equation}
We change the measure from $\mathbf{P}$ to a new probability measure $\mathbf{\overline{P}}$ by Ladon-Nikodym Theorem such that $\mathcal{L}_{\mathbf{\overline{P}}}((\overline{X})_{1\leq n\leq N})=\mathcal{L}_{\mathbf{P}}((X)_{1\leq n\leq N})$. Then for any $\delta>0$, we can find a constant $C>0$, such that for any $k\in\{-M,\ldots,M\}$ and large enough integer $N$, we have
\begin{equation}\label{524}
\frac{\rm{d}\mathbf{P}}{\rm{d}\mathbf{\overline{P}}}\geq C(|k|N)^{-(\alpha+1+\delta)}.
\end{equation}
\end{lemma}
\begin{proof}[Proof of Lemma \ref{T502}]
It is obvious that $\mathbf{\overline{P}}$ and $\mathbf{P}$ only differ on the distribution of $X_{\frac{N}{2}+1}$. We use the notations
\begin{equation}\label{525}
\mathbf{P}\left(X_{\frac{N}{2}+1}=x\right)=p_{x},~\forall x\in\mathbbm{Z}
\end{equation}
and $h=h_{N}(n,k)$ for short. Then
\begin{equation}\label{526}
\mathbf{\overline{P}}\left(X_{\frac{N}{2}+1}=x\right)=\mathbf{\overline{P}}\left(\overline{X}_{\frac{N}{2}+1}=x-h\right)=p_{x-h}.
\end{equation}
The Radon-Nikodym derivative can be written explicitly by
\begin{equation}\label{527}
\frac{\rm{d}\mathbf{P}}{\rm{d}\mathbf{\overline{P}}}=\sum\limits_{x\in\mathbbm{Z}\setminus\{0,h\}}\mathbbm{1}_{\left\{X_{\frac{N}{2}+1}=x
\right\}}\frac{p_{x}}{p_{x-h}}+\mathbbm{1}_{\left\{X_{\frac{N}{2}+1}=0\right\}}\frac{p_{0}}{p_{-h}}+\mathbbm{1}_{\left\{X_{\frac{N}{2}+1}=h
\right\}}\frac{p_{h}}{p_{0}}.
\end{equation}
The summand in the summation on the right hand side of \eqref{527} is
\begin{equation}\label{528}
\frac{L(|x|)}{L(|x-h|)}\left|1-\frac{h}{x}\right|^{\alpha+1}.
\end{equation}
By Potter's bound (see \cite[Theorem 1.5.6]{bingham1989regular}), given a slowly varying function $L(x)$, for any $\delta>0$, $A\geq1$, there exists some constant $C=C(\delta,A)$, such that
\begin{equation}\label{529}
\frac{L(x)}{L(y)}\geq C\min\limits_{x\geq A,y\geq A}\left\{\left(\frac{x}{y}\right)^{\delta},\left(\frac{x}{y}\right)^{-\delta}\right\}.
\end{equation}
We can partition the summation range by $(-\infty,1],[1,h-1],[h+1,\infty)$ to get rid of the absolute value in \eqref{528} and then apply \eqref{529} to achieve \eqref{524}. Note that the term $(\log N)^{2}$ in $L$ can be ignored by some adjustment in the power $\delta$, since it is a slowly varying function.
\end{proof}
Now by \eqref{522} and Lemma \ref{T502}, for $\delta=\frac{\epsilon}{2}$ and some constant $C$, we have
\begin{equation}\label{530}
\begin{split}
Z_{N,\beta}^{\omega}(k)&=\overline{E}\left[\frac{\rm{d}\mathbf{P}}{\rm{d}\mathbf{\overline{P}}}\exp\left(\beta\sum\limits_{n=1}^{N}
\omega_{n,\overline{S}_{n}+h_{N}(n,k)}\right)\mathbbm{1}_{\left\{\overline{S}_{n}\in I_{N}^{0},~\forall n\in\left[\frac{N}{2}+1,N\right]\right\}}\right]\\
&\geq C(MN)^{-\alpha+1+\frac{\epsilon}{2}}\overline{Z}_{N,\beta}^{\omega}(k),
\end{split}
\end{equation}
where in the last inequality, we use the property that $\mathcal{L}_{\mathbf{\overline{P}}}((\overline{X})_{1\leq n\leq N})=\mathcal{L}_{\mathbf{P}}((X)_{1\leq n\leq N})$. Combine \eqref{521} and \eqref{530} and then we finish the proof of Proposition \ref{T501}.
\end{proof}

\begin{remark}\label{Ts3}
In \cite{lacoin2011influence}, the author also showed that for a Brownian polymer $B_{t}$ in a continuous Gaussian field, $B_{t}$ cannot fluctuate on a scale larger than $\mathcal{O}(N^{\frac{3}{4}})$. However, in Theorem \ref{T116}, we have shown that if the one step distribution of the random walk has polynomial decay, then even though it is in the domain of attraction of the Gaussian law, it will fluctuate on a scale larger than $\mathcal{O}(N^{1-\epsilon})$ for arbitrarily small $\epsilon>0$, which is much larger than $N^{\frac{3}{4}}$. This is a remarkable difference between the long-range model and the short-range model, which is comparable to the nearest-neighbor model.
\end{remark}

\textbf{Acknowledgements:}~I would like to acknowledge support from AcRF Tier 1 grant R-146-000-220-112. I am deeply indebted to my supervisor Prof.\ Rongfeng Sun, who introduced this topic to me, provided suggestions and discussed with me and I am grateful to my classmate Jinjiong Yu for helpful discussion. I also want to thank two referees, whose comments help me add the super-$\alpha$-stable result and correct some mistakes, which improve an early version of this paper.
\bibliographystyle{plain}
\bibliography{references}
\end{document}